\input ./style/arxiv-ba.cfg
\documentclass[ba,linksfromyear,preprint]{imsart}
\makeatletter
   \@ifpackageloaded{natbib}{}{\usepackage{natbib}}
\makeatother
\usepackage{mathrsfs}
\usepackage{ifthen}

\pubyear{2015}
\volume{10}
\issue{4}
\firstpage{759}
\lastpage{789}
\doi{10.1214/15-BA941}% Updated by VTEXPTS2LaTeX.exe, 30.01.2015 10:16
\startlocaldefs

\newtheorem{theorem}{Theorem}
\newtheorem{corollary}{Corollary}
\newtheorem{lemma}{Lemma}
\newtheorem{proposition}{Proposition}
\theoremstyle{definition}
\newtheorem{defn}{Definition}
\newtheorem{example}{Example}
\newtheorem{assumption}{Assumption}
\theoremstyle{remark}
\newtheorem{remark}{Remark}

\newcommand{\FF}{\mathbb{F}}
\newcommand{\YY}{\mathbb{Y}}
\newcommand{\Fconv}{\left\langle\mathbb{F}_0\right\rangle}
\newcommand{\Y}{\mathscr{Y}}

\newcommand{\E}{\mathsf{E}}
\newcommand{\prob}{\mathsf{P}}

\newcommand{\unif}{{\sf Unif}}

\newcommand{\eps}{\varepsilon}
\renewcommand{\phi}{\varphi}

\newcommand{\ftrue}{f_0}
\newcommand{\fbest}{f^\star}
\newcommand{\fbestn}{f^{\star(n)}}
\newcommand{\fhat}{\hat f}
\newcommand{\fn}{f^{(n)}}
\newcommand{\fhatn}{\hat f^{(n)}}
\newcommand{\ftruen}{f_0^{(n)}}

\newcommand{\Kstar}{K^\star}
\newcommand{\hstar}{h^\star}
\endlocaldefs

\begin{document}

\begin{frontmatter}

\title{On Posterior Concentration in Misspecified Models}
\runtitle{On Posterior Concentration in Misspecified Models}

\begin{aug}
\author[addr1]{\fnms{R.~V.} \snm{Ramamoorthi}\ead[label=e1]{ramamoor@stt.msu.edu}}
\author[addr2]{\fnms{Karthik} \snm{Sriram}\corref{}\ead[label=e2]{karthiks@iimahd.ernet.in}},
\and
\author[addr3]{\fnms{Ryan} \snm{Martin}\ead[label=e3]{rgmartin@uic.edu}}

\runauthor{R.~V. Ramamoorthi, K. Sriram, and R. Martin}

\address[addr1]{Michigan State University, East Lansing, Michigan, USA, \printead{e1}}
\address[addr2]{Indian Institute of Management Ahmedabad, India, \printead{e2}}
\address[addr3]{University of Illinois at Chicago, USA, \printead{e3}}

%\thankstext{t1}{Some comment}
%\thankstext{t2}{First supporter of the project}
%\thankstext{t3}{Second supporter of the project}

\end{aug}

%% Abstract %%
%
\begin{abstract}
We investigate the asymptotic behavior of Bayesian posterior
distributions under independent and identically distributed ($i.i.d.$)
misspecified models. More specifically, we study the concentration of
the posterior distribution on neighborhoods of $\fbest$, the density
that is closest in the Kullback--Leibler sense to the true model $f_0$.
We note, through examples, the need for assumptions beyond the usual
Kullback--Leibler support assumption. We then investigate consistency
with respect to a general metric under three assumptions, each based on
a notion of divergence measure, and then apply these to a weighted
$L_1$-metric in convex models and non-convex models.

Although a few results on this topic are available, we believe that
these are somewhat inaccessible due, in part, to the technicalities and
the subtle differences compared to the more familiar well-specified
model case. One of our goals is to make some of the available
results, especially that of \citet{Kleijn_van2006}, more accessible.
Unlike their paper, our approach does not require construction of test
sequences. We also discuss a preliminary extension of the $i.i.d.$
results to the independent but not identically distributed ($i.n.i.d.$)
case.
\end{abstract}

%% Keywords %%

\setattribute{keyword}{MSC}{MSC 2010:}


\end{frontmatter}

%% Mainmatter %%

%s1 ###
\section{Introduction}
Let $\FF_0$ be a family of densities with respect to a $\sigma
$-finite measure on a measure space. The object of study is the
posterior distribution arising out of the model which consists of a
prior distribution $\Pi$ on $\FF_0$ and for any given $f\in\FF_0$,
$Y_{1:n} = (Y_1, Y_2, \ldots, Y_n)$ are independent and identically
distributed ($i.i.d.$) as $f$. We investigate the behavior of the
posterior distribution when the ``true'' model $f_0$ is not necessarily
in $\FF_0$. The posterior is typically expected to concentrate around
a density $\fbest$ in $\FF_0$ that minimizes the Kullback--Leibler
divergence from $\ftrue$.

An early investigation of this set up goes back to \cite{Berk1966}. An
extensive study of parametric model appears in \cite
{Bunke_milhaud1998}. \cite{Lee} investigate concentration of the
posterior and its behavior in testing problems when the prior is on an
exponential model. The infinite-dimensional nonparametric case has been
studied by \citet{Kleijn_van2006} and \citet{DW2013} for the $i.i.d.$
case, and \citet{shalizi2009} for the non-$i.i.d.$ case.

For the nonparametric $i.i.d.$ case studied in this note, \cite
{Kleijn_van2006} is the basic paper. The standard approach to this
problem is to first identify sets whose posterior probability goes to 0
and then relate these to the topology of interest. In their paper,
\citeauthor*{Kleijn_van2006} develop both these aspects together and,
in addition to consistency, also develop rates. \cite{DW2013} take a
somewhat different route towards providing sufficient conditions,
specifically for Hellinger-consistency.

Our starting point is \cite{Kleijn_van2006}. To help motivate our
approach, we first summarize key steps in their work. Let $\E_0$
denote expectation with respect to~$\ftrue$.\looseness=1
\begin{itemize}
\item They start with the Kullback--Leibler support assumption on the
prior $\Pi$, i.e., \\
%
%e1 ###
\begin{equation} \label{kl.assn}
\Pi\left( f: \E_0\log\frac{\fbest}{f} < \eps\right) >0 , \mbox
{ for any } \eps>0.
\end{equation}
\item They cover the sets $S_j=\{f:\ j\eps\leq d(f,\fbest)<(j+1)\eps
\}$, for $j\geq1$, by convex sets $A$ satisfying:
%
%e2 ###
\begin{equation}\label{eqn0}
\sup_{f \in A} \inf_{0 \leq\alpha\leq1} \E_0 \left(\frac
{f}{\fbest}\right)^\alpha< e^{-j^2\eps^2/4}.
\end{equation}
\item Then they show that, if the assumption in (\ref{kl.assn}) is
satisfied, posterior probability of sets satisfying (\ref{eqn0}) goes
to $0$ by constructing a sequence of exponential tests for a testing
problem that involves non-probability measures. Then, based on the
number of sets satisfying (\ref{eqn0}) required to cover $S_j$, they
develop a notion of entropy for testing problems of a set $S_j$. When
such entropy can be controlled suitably, it is shown that posterior
probability of $\{d(f,\fbest)\geq\eps\}=\cup_{j\geq1}S_j$ goes to 0.
\end{itemize}
In this paper, we first provide a simple proof to show that, if the
assumption in (\ref{kl.assn}) is satisfied, then probability of sets
satisfying (\ref{eqn0}) goes to 0. Our proof does not involve testing
problems. We further observe that for a given convex set, the condition
in (\ref{eqn0}) is, in fact, equivalent to a simpler condition
based on Kullback--Leibler divergence.

Consistency is related to the topology on the space of densities,
usually the weak topology or the Hellinger-metric topology. Towards
this, we give two examples in the appendix that point out the need for
additional assumptions beyond requiring that $\fbest$ be in the
topological support or Kullback--Leibler support of $\Pi$. In order to
gain insight, we first study consistency with respect to a general
metric under a set of three assumptions, each based on a notion of
divergence. The first assumption is based on Kullback--Leibler
divergence, the second is based on (\ref{eqn0}), and the third is
based on a relatively simpler notion. We show that for a weighted
$L_1$-metric, such assumptions hold in convex models or when the
specified family is compact. The first assumption mainly works for
compact and convex families. The second assumption along with an
appropriate metric entropy condition gives consistency for convex
families. As a consequence, we derive a consistency result (Theorem
\ref{T:analogKV}), which is analogous to \cite{Kleijn_van2006}. The
third assumption is useful for non-convex (e.g., parametric) models. In
this case, we circumvent\vadjust{\eject} the convexity requirement by making a
continuity assumption on the likelihood ratio and show posterior
consistency under an appropriate prior-summability or metric-entropy
condition. As a particular consequence, Theorem~\ref{DWKVequi} gives
Hellinger consistency analogous to \cite{DW2013}.

We believe that our methods are simple and transparent, and provide
useful insights on the requirements for the metric, while also making
some of the results in \cite{Kleijn_van2006} more accessible. As
another small difference, we note that our consistency results are
presented in the \xch{`}{'}{\it almost sure}' sense, as compared to convergence
of means. We also look at one immediate extension of the $i.i.d.$
results to independent but not identically distributed ($i.n.i.d.$)
models. We note that our study in the $i.n.i.d.$ case is preliminary
and is presented as an initial approach.

The remainder of the paper is organized as follows. Section~\ref
{S:setup} sets our notation and provides some basic results.
Section~\ref{S:general} presents the consistency results for a general
metric. Section~\ref{S:L1weak} presents some important results
specific to $L_1$ and weak topologies. Section~\ref{S:examples}
contains examples to demonstrate the application of our results.
Finally, Section~\ref{S:inid} discusses an extension of the $i.i.d.$
results to the $i.n.i.d.$ case. In the interest of flow, supporting
results and details of some proofs are included in the appendix.

%s2 ###
\section{Notations and preliminary results}
\label{S:setup}
%s2.1 ###
\subsection{Notations}
\label{S:notation}
Let $Y_{1:n}=(Y_{1}, Y_{2}, \dots, Y_{n} )$ be an
$i.i.d.$ sample from an unknown ``true\xch{''}{"} density $f_{0}$ with respect to
a $\sigma$-finite measure $\mu$ on a measure space $(\YY, \Y)$.
$\FF_0$ is a family of density functions specified to model $Y_{1:n}$.
$\ftrue$ is not necessarily in $\FF_0$. Let $\Pi$ be a prior on $\FF
_0$. We let $\prob_0$ and $\E_0$ denote probability and expectation
with respect to $\ftrue$. When talking about joint distribution of
finite or infinite $i.i.d.$ sequences with respect to $\prob_0$, we
will omit the superscript in $\prob_0^n$ or $\prob_0^\infty$.

It is well known that the posterior typically concentrates around a
density that minimizes the Kullback--Leibler divergence from $\ftrue$,
given by
\begin{equation*}
\label{KLdiv0}
K(\ftrue, f):= \E_0 \log\frac{\ftrue}{f}=\int\log\frac{\ftrue
}{f} \ftrue d\mu.
\end{equation*}
Accordingly, we assume that there is a fixed unique $\fbest\in\FF_0$
such that
\begin{equation*}
\label{KLdiv}
K(\ftrue, \fbest) = \inf_{f\in\FF_0} K(\ftrue,f).
\end{equation*}
For any density $f$, we define
\begin{equation*}
\label{KLdiff}
\Kstar(\ftrue,f):= K(\ftrue,f)- K(\ftrue,\fbest).
\end{equation*}
We assume throughout that $\int(f / \fbest) \ftrue\,d\mu< \infty$
for all $f \in\FF_0$, and also $ \int(\ftrue/ \fbest) \,d\mu<
\infty$. The latter condition is useful since we will later (Section
\ref{S:L1weak}) consider weak and $L_1$ topologies with respect to the
measure $\mu_0$, where $d\mu_0 = (\ftrue/ \fbest) \,d\mu$. The
weighted $L_1$ metric, $L_1(\mu_0)$, appears to be a natural choice in
misspecified models, as opposed to the usual $L_1(\mu)$ metric.

Let $\Fconv$ be the smallest convex set containing $\FF_0$. In this
note, a convex set is one that is closed under mixtures. That is, a
general subset $A \subseteq\Fconv$ is called convex if, for any
probability measure $\nu$ on $A$, the mixture $\fhat_\nu:= \int_A f
\,\nu(df)$ belongs to $A$. It is convenient to extend $\Pi$ to
$\Fconv$ by defining $ \Pi(A) := \Pi(A \cap\FF_0)$, for any
measurable subset $A$ of $\Fconv$. Note that we do not assume that
$\fbest$ minimizes the Kullback--Leibler divergence in $\Fconv$. In
addition, we define:
\begin{eqnarray}
\hstar_{\alpha}(\ftrue,f) &:=& \E_0\left(f/\fbest\right)^\alpha
= \int\left(f/\fbest\right)^\alpha\ftrue d\mu, \nonumber\\
\fn&:=&\fn(y_1, y_2,\dots, y_n):= \prod_{i=1}^{n}f(y_{i}),\nonumber
\\
\fhatn_{\nu}&:=&\int\fn d\nu(f), \mbox{ where } \nu\mbox{ is a
probability measure on } \Fconv,\nonumber\\
\fhatn_{A}&:=&\fhatn_{\Pi_{A}} , \mbox{ where } \Pi_{A}(\cdot
):=\Pi(A \ \cap\ \cdot\ )/\Pi(A), \nonumber\\
\hstar_{\alpha}(\ftruen,\fn)&:=& \E_0\left(\fn/\fbestn\right
)^{\alpha}=\int\left(\fn/\fbestn\right)^\alpha\ftruen d\mu.
\nonumber
%&& \hstar(\ftruen, \fn):= \inf_{0\leq\alpha\leq1} \hstar_{\alpha}(
%\ftruen,\fn).
\end{eqnarray}
Finally, we write down the formula for the posterior distribution as
%
%e3 ###
\begin{equation}
\label{posterior}
\Pi(A| Y_{1:n}) = \Pi(A) \, \frac{\fhatn_A(Y_1,\ldots,Y_n) /
\fbestn(Y_1,\ldots,Y_n)}{\fhatn_\Pi(Y_1,\ldots,Y_n) / \fbestn
(Y_1,\ldots,Y_n)}.
\end{equation}

%s2.2 ###
\subsection{Preliminary results}
\label{S:prelim}
We start with
\begin{assumption}
\label{as:prior.concentration}
$\forall\ \eps>0$, $\Pi\left(f:\ \Kstar(\ftrue, f) <\eps\right)>0$.
\end{assumption}
The following proposition, which helps handle the denominator of (\ref
{posterior}), is the main consequence of Assumption 1.
\begin{proposition}\label{post.denom}
If Assumption~\ref{as:prior.concentration} holds, then, for any $\beta
> 0$,
\begin{equation*}
\liminf_{n \rightarrow\infty} e^{n \beta}\cdot\fhatn_\Pi
(Y_1,\ldots,Y_n) / \fbestn(Y_1,\ldots,Y_n) = \infty\qquad \prob
_0\mbox{-a.s\xch{.}{..}}
\end{equation*}
\end{proposition}
The proof of the proposition is along the lines of Lemma 4.4.1 in \cite
{Ghosh_RVR2003}.

In view of Proposition~\ref{post.denom}, $\Pi(A|Y_{1:n}) \rightarrow
0 \mbox{ }\ \prob_0$-a.s., if it can be ensured that
for some $\beta_0 >0$,
%
%e4 ###
\begin{equation}\label{equn2}
e^{n \beta_0} \cdot\Pi(A)\cdot\frac{\fhatn_{A}(Y_{1:n})}{\fbestn
(Y_{1:n})} \rightarrow0 \qquad\prob_0\mbox{-a.s\xch{\ldots}{..} }
\end{equation}
Towards handling (\ref{equn2}), we work with three notions of
divergence of $f$ from $\fbest$:
\begin{enumerate}
\item[] (i) $ \Kstar(\ftrue, f),$ based on Kullback--Leibler divergence,
\item[] (ii) $ (1- \inf_{0\leq\alpha\leq1}\hstar_{\alpha
}(\ftrue,f))$, based on \cite{Kleijn_van2006} and
\item[] (iii) $(1- \hstar_{\alpha_0}(\ftrue,f))$ for
some $0 < \alpha_0 <1$, a notion relatively simpler than (ii).
\end{enumerate}
The proposition below describes the relationship between these. The
second condition in the proposition was introduced by
\xch{\citet*{Kleijn_van2006}}{\citeauthor{Kleijn_van2006}}.
\begin{proposition}\label{KVKLrelation}
Consider the following three conditions for a subset $A$:
\begin{enumerate}
\item[(i)] For some $\epsilon>0$, $ \inf_{f \in A} \Kstar(\ftrue,
f) > \epsilon$.
\item[(ii)] For some $\delta>0$, $ \sup_{f \in A} \inf_{0\leq
\alpha\leq1} \hstar_\alpha(\ftrue, f) < e^{-\delta} $.
\item[(iii)] For some $0<\alpha_0<1$ and $\eta>0$, $\sup_{f \in A}
\hstar_{\alpha_0}(\ftrue,f) < e^{-\eta} $.
\end{enumerate}
For any set $A$, $ (iii) \Rightarrow(ii) \Rightarrow(i)$. Further, if
the set $A$ is convex, then they are all equivalent.
\end{proposition}
The proof of Proposition~\ref{KVKLrelation} uses the minimax theorem
and is provided in Appendix~\ref{proofs}. The easy proposition below
plays a central role.
\begin{proposition}\label{prop.convexA}
Suppose $A\subset\Fconv$ is convex. If for some $0<\alpha<1$ and
$\delta>0$,
\begin{eqnarray*}
\sup_{f\in A} \hstar_{\alpha}(\ftrue,f) \leq e^{-\delta}\label{halpha},
\end{eqnarray*}
then for any probability measure $\nu$ on $A$,
\begin{eqnarray*}
\hstar_{\alpha}(\ftruen,\fhatn_\nu) \leq e^{-n\delta}.
\end{eqnarray*}
\end{proposition}
\begin{proof}
The result follows by the use of convexity and induction. Here is an outline.
When $ n =1$, the claim holds by convexity of $A$. \\ When $ n=2$,
$f_\nu^{(2)} (y_1,y_2)$ is the marginal density of $ Y_1,Y_2$ under
the model: $ Y_1,Y_2 |f \stackrel{iid}{\sim} \ f$ and $ f \sim\nu
(\cdot)$. Write
\[
\left[\frac{f_\nu^{(2)}(y_1,y_2)}{f^{\star(2)}(y_1,y_2)}\right
]^\alpha= \left[\frac{f_{\nu}(y_1|y_2)}{\fbest(y_1)}\right
]^\alpha\left[ \frac{f_{\nu}(y_2)}{\fbest(y_2)}\right]^\alpha
\]
where the first term inside the brackets on the right-hand side,
$f_{\nu}(y_1|y_2),$ is the conditional density of $Y_1$ given $Y_2$,
and the second term, $f_{\nu}(y_2),$ is the marginal density of $
Y_2$, obtained from the joint density $f_\nu^{(2)}(y_1,y_2)$. By
convexity of $A$, for all $y_2$, $f_\nu(\cdot|y_2) \in A $. Hence, we have
\[
\E_0 \left[\left(\frac{f_\nu(y_1|y_2)}{\fbest(y_1)}\right
)^\alpha\Big{|} y_2\right] \leq\ e^{-\delta}.
\]
Since $f_\nu(y_2)\in A$, $\E_0[ \frac{f_{\nu}(y_2)}{\fbest
(y_2)}]^\alpha\leq e^{-\delta}$. Therefore,
\[
\E_0 \left[\frac{f_\nu(y_1,y_2)}{\fbest(y_1,y_2)}\right]^\alpha
\leq e^{-2\delta}.
\]
A similar induction argument for general $n$ completes the proof.
\end{proof}

\begin{theorem} \label{thm.convexA}
Suppose Assumption 1 holds. If $ A \subset\Fconv$ is convex and
satisfies $(i)$ (or equivalently, $(ii)$ or $(iii)$) of Proposition
\ref{KVKLrelation}, then
\[
\Pi(A | Y_{1:n}) \rightarrow0 \qquad\prob_0\mbox{-a.s\xch{.}{..}}
\]
\end{theorem}
\begin{proof} Suppose $ \sup_{f \in A} \hstar_\alpha(\ftrue,f) =
\sup_{f \in A} \E_0 (f/\fbest)^\alpha< e^{-\delta} $
for some $0<\alpha<1$ and $\delta>0$. Then, by Proposition~\ref
{prop.convexA}\xch{,}{ ,} $\hstar_\alpha(\ftruen, \fhatn_A)\leq e^{-n\delta
}$. Let $2\beta_0 < \delta$.
Then
\begin{align*}
&\prob_0\left(\int_A \prod_{i=1}^n \frac{f (Y_i)}{\fbest(Y_i)}
d\Pi(f) > e^{ -2n \beta_0 } \right)\\
& =\prob_0\left(\left(\int_A \prod_{i=1}^n \Pi(A)\frac{f
(Y_i)}{\fbest(Y_i)} d\Pi_A (f)\right)^\alpha> e^{ - 2n\alpha \beta
_0 } \right) \\
& \leq\Pi^\alpha(A) \cdot\hstar_\alpha(\ftruen, \fn_A)\cdot
e^{2n \beta_0 \alpha}
\end{align*}
Hence
\[
\prob_0\left\{\int_A \prod_{i=1}^n \frac{f (Y_i)}{\fbest(Y_i)} d\Pi
(f) > e^{ -2n \beta_0 } \right\} \leq e^{- n (\delta-2\beta_0 ) }.
\]
Since the expression on the right-hand side is summable, we observe by
using Borel--Cantelli lemma that (\ref{equn2}) is satisfied. This
observation, along with Proposition~\ref{post.denom}, gives the result.
\end{proof}
%
%s3 ###
\section{Consistency with general metric}
\label{S:general}
Consistency requires the posterior to concentrate on neighborhoods of
$\fbest$ with respect to some metric $d$. In developing conditions for
consistency with respect to $d$, we encounter a few issues.

First, a necessary condition is that $\fbest$ be in the topological
support of $\Pi$ with respect to this metric. Assumption 1 by itself
does not ensure this. We present two examples in Appendix~\ref
{illust.eg} to illustrate this and point out the need for stronger
assumptions. The first example shows that while the presence of $\fbest
$ in the $L_1$ support of $\Pi$ is necessary for consistency, this is
not automatically guaranteed by the positivity of Kullback--Leibler
neighborhoods specified in Assumption 1. The next example demonstrates
that the presence of $\fbest$ in the $L_1$ support and Assumption 1 by
themselves are not enough to ensure consistency.

Second, since the complement of a $d$-neighborhood is not convex in
general, the equivalence in Proposition~\ref{KVKLrelation} is
inapplicable. One approach is to suitably cover the complement by
$d$-balls. This in turn requires that each ball satisfy one of the
three conditions in Proposition~\ref{KVKLrelation}. Motivated by
these, we investigate consequences of each of the following set of assumptions.
\begin{itemize}
\item[{}]{\bf Assumption 2a.} Every neighborhood $U= \{f \in\Fconv:
d(f,\fbest) < \epsilon\}$ contains a set of the form $\{f\in\Fconv:
\ \Kstar(\ftrue,f)<\delta\}$ for some $\delta>0$.
\item[] {\bf Assumption 2b.} Every neighborhood $U= \{f \in\Fconv:
d(f,\fbest) < \epsilon\} $ contains a set of the form $\{f\in\Fconv
: \inf_{0\leq\alpha\leq1} \hstar_{\alpha} (\ftrue,f) >
e^{-\delta}\}$ for some $\delta>0$.
\item[] {\bf Assumption 2c.} Every neighborhood $U= \{f \in\FF_0:
d(f, \fbest) < \epsilon\} $ contains a set of the form $\{f\in \FF
_0: \hstar_{\alpha_0}(\ftrue,f) > e^{-\delta}\}$ for some $0<\alpha
_0<1$ and $\delta>0$.
\end{itemize}

We note that Assumptions 2a and 2b are stated in terms of the
convexification $\Fconv$ of $\FF_0$, whereas Assumption 2c is stated
in terms of $\FF_0$. The presence of $\Fconv$ makes it hard to verify
the first two assumptions in non-convex models. However, we study the
consequences of these assumptions in a general metric space because of
the insight it provides into the requirements on the metric for
consistency and shows the usefulness of Assumption 2c in non-convex
models. In the next section, we discuss sufficient conditions for
Assumptions 2a, 2b, and 2c, with respect to $L_1$ and weak topologies.

For the rest of this section, we study posterior consistency based on
each of these assumptions. We find that Assumption 2a is mainly useful
when $\FF_0$ is convex and compact, Assumption 2b is useful when $\FF
_0$ is convex but may not be compact, and Assumption 2c is useful when
the family is neither convex nor compact.

The following theorem based on Assumption 2a is an easy consequence of
Theorem~\ref{thm.convexA}.
\begin{theorem}\label{KL.compact}
Let $d$ be a metric such that $d$-balls in $\Fconv$ are convex sets
and $\FF_0$ is compact with respect to $d$. Let $U = \{f\in\Fconv:
d(f,\fbest)< \epsilon\}$. Suppose Assumptions 1
and 2a hold. Then
\[
\Pi( U^{c} | Y_{1:n}) \rightarrow1 \qquad\prob_0\mbox{-a.s\xch{.}{..} }
\]
\begin{proof}
By Assumption 2a, let $\{f\in\Fconv: \Kstar( \ftrue, f) < \delta\}
\subset \{f\in\Fconv: d(f,\fbest)< \epsilon/2 \}$. Since $U^c\cap
\FF_0$ is compact, it can be covered by $B_1, B_2, \ldots, B_k$ all
contained in $\Fconv$ with $ B_ i = \{f \in\Fconv: d(f,f_i) <
\epsilon/3 \}$ for some $f_1,f_2, \dots, f_k \in\FF_0$. Each of
these balls is convex and disjoint from $\{f\in\Fconv: d(f,\fbest)<
\epsilon/2 \}$. Assumption 2a ensures that each $B_i$ satisfies
property (i) of Proposition~\ref{KVKLrelation}. Since there are
finitely many such sets, the result follows using Theorem~\ref{thm.convexA}.
\end{proof}
\end{theorem}
In the proof of Proposition~\ref{KVKLrelation} provided in Appendix
\ref{proofs}, it is clear that the choice of $\alpha_0$ and $\eta$
made while establishing equivalence of conditions depends on the
specific set $A$. Hence, it does not appear that the approach based on
Assumption 2a can be carried easily beyond convex and compact families.
Below, we take an approach based on Assumption 2b, which is more in
line with \cite{Kleijn_van2006}. Theorem~\ref{thm.convexA} derives
posterior consistency for convex sets. Since complement of a
$d$-neighborhoods will not be convex in general, the approach here is
to cover with a suitable number of convex sets with diminishing
posterior probabilities. Towards this end, we make the following assumption.
\begin{itemize}
\item[] {\bf Assumption 3.} There exist subsets $\{V_n, W_n\}_{n \geq
1} $ such that
$\FF_0 \subseteq V_n \cup W_n$ and
\item[] (a) $\Pi( W_n) < e^{-n\Delta}$ for some $\Delta>2\E_0\log
\frac{f_0}{f_*}$.
\item[] (b) For every $\epsilon>0$, $ {V}_n$ can be covered by $J_n $
$d$-balls in $\Fconv$ of radius less than $\epsilon$, where $J_n$ is
a polynomial in $n$, i.e., for some $r>0$, $J_n \leq a n^r$.
\end{itemize}

The simple lemma below will be useful to derive the results that follow.
\begin{lemma}
\label{lem2new}
Let $T_i, i =1,2 ,\ldots, k$ be non-negative random variables. Then
\[
\prob\left( \sum_{i=1}^{k} T_i > e^{-\epsilon}\right) \leq
e^\epsilon\sum_{i=1}^{k}\inf_{0 \leq\alpha\leq1} \E T_i^\alpha.
\]
\begin{proof} The result follows since, $ \prob( \sum_{i=1}^{k} T_i >
e^{-\epsilon}) = \prob( \sum_{i=1}^{k}\min(T_i, 1) > e^{-\epsilon
}) $ and $ \min(T_i , 1 ) \leq T_i^{\alpha_i}$ for any $0<\alpha
_i<1$. Consequently,
\begin{align*}
\prob\left( \sum_{i=1}^{k} T_i > e^{-\epsilon}\right)\leq\prob
\left( \sum_{i=1}^{k} T_i ^{\alpha_i} > e^{-\epsilon}\right) \leq
e^\epsilon\sum_{i=1}^{k} \E T_i ^{\alpha_i} .
\end{align*}

Taking the infimum over $\alpha_1, \alpha_2, \ldots,\alpha_k$ on
both sides, we get the result.
\end{proof}
\end{lemma}
We now derive posterior consistency under Assumptions 2b and 3,
followed by a result that is analogous to the posterior consistency
result in \cite{Kleijn_van2006}.
\begin{theorem}
\label{dconsis}
Let metric $d$ be such that $d$-balls in $\Fconv$ are convex sets and
let $U_\epsilon=\{f\in\Fconv: d(f,\fbest) < \epsilon\} $. If
Assumptions 1, 2b and 3 hold, then
\[
\Pi(U_{\epsilon}^c | Y_{1:n} ) \rightarrow0\qquad\prob_0 \mbox{-a.s\xch{.}{..}}
\]

\begin{proof} Let $U_{\epsilon/2} = \{f \in\Fconv: d(\fbest, f) <
\epsilon/2\} $, and as guaranteed by Assumption 2b, let
\[
\{f\in\Fconv: \inf_{0\leq\alpha\leq1} \hstar_{\alpha} (\ftrue
,f) > e^{-\delta}\} \subset U_{\epsilon/2} .
\]
Let $A_1,A_2, \dots, A_{J_n}$ be open $d$-balls of radius $\epsilon
/3$ that cover $ U_{\epsilon}^c \cap V_n$. Then,
\begin{align*}
\prob_0 \left( \int_{U_{\epsilon}^c \cap V_n} \fn/\fbestn d\Pi >
e^{-n\beta} \right) &\leq\prob_0\left( \sum_{i=1}^{J_n} \int
_{A_i} \fn/\fbestn d\Pi > e^{-n\beta} \right)\\
& \leq\prob_0\left( \sum_{i=1}^{J_n} \int_{A_i} \fn/\fbestn d\Pi
_i > e^{-n\beta} \right),
\end{align*}
where $ \Pi_i$ is $ \Pi$ normalized to 1 on $A_i $. Set $T_i = \int
_{A_i} \fn/\fbestn d\Pi_i$ so that the expression on the right-hand
side of the last inequality can be written as:
\[
\prob_0\left( \sum_{i=1}^{J_n} T_i > e^{-n\beta} \right).
\]
By Lemma~\ref{lem2new},
\[
\prob_0\left( \sum_{i=1}^{J_n} T_i > e^{-n\beta} \right) \leq e^{n
\beta}\sum_i\inf_{0 \leq\alpha\leq1} \E_0 T_i^\alpha.
\]
Since $A_i$ does not intersect $U_{\epsilon/2}$, $\inf_{0\leq\alpha
\leq1}\E_0 (f/\fbest)^{\alpha}< e^{-\delta}$. By
Proposition~\ref{prop.convexA}, for each~$i$, $ \inf_{0 \leq\alpha
\leq1} \E T_i^\alpha< e^{-n \delta}$, so that
\[
\prob_0 \left( \int_{U_{\epsilon}^c \cap V_n} \fn/\fbestn d\Pi >
e^{-n\beta} \right)< e^{n\beta} n^r e^{-n\delta}.
\]
A choice of small enough $\beta$ and an application of Borel--Cantelli
lemma with $\beta_0<\beta$ gives
\[
e^{n\beta_0}\int_{U_{\epsilon}^c \cap V_n} \fn/\fbestn d\Pi
\rightarrow\ 0 \qquad\prob_0\mbox{-a.s\xch{.}{..}}
\]
This, along with Proposition~\ref{post.denom}, gives $\Pi(
U_{\epsilon}^c \cap V_n | Y_{1:n})\rightarrow\ 0 \quad \prob
_0\mbox{-a.s\xch{.}{..}}$

As for $W_n$, first an argument in the lines of Lemma 4.4.1 of \cite
{Ghosh_RVR2003} can be used to conclude that for any $\beta>\E_0\log
\frac{\ftrue}{\fbest}$,
%
%e5 ###
\begin{eqnarray}
\liminf_{n \rightarrow\infty} e^{n \beta}\int_{\Fconv} \fn
/\ftruen d\Pi= \infty\qquad \prob_0\mbox{-a.s\xch{.}{..}}\label{Wndenom}
\end{eqnarray}
Then, for $\Delta> 2\E_0\log\frac{\ftrue}{\fbest}$, an
application of Markov's inequality gives
%
%e6 ###
\begin{eqnarray}
&& \prob_0 \left( \int_{W_n} \fn/\ftruen d\Pi > e^{-n\frac{\Delta
}{2}} \right) \nonumber\\
&&\leq e^{n\frac{\Delta}{2}} \cdot\int_{W_n} \E_0\left(\fn
/\ftruen\right) d\Pi= e^{n\frac{\Delta}{2}} \Pi(W_n) \leq
e^{-n\frac{\Delta}{2}}.\label{Wnnumer}
\end{eqnarray}
Equations \eqref{Wndenom} and \eqref{Wnnumer} together imply that
$\Pi(W_n|Y_{1:n})\rightarrow\ 0 \quad\prob_0\mbox{-a.s\xch{.}{..}}$
\end{proof}
\end{theorem}
The approach taken in Theorem~\ref{dconsis} can be adapted to derive a
result that is analogous to Corollary 2.1 of \cite{Kleijn_van2006}.
The entropy condition in their paper assumes that each set $S_j=\{f\in
\FF_0: j\epsilon\leq d(f,\fbest) < (j+1) \epsilon\}$ can be covered
by $N_j$ convex sets $B_k$ with the property $\sup_{f\in B_k}\inf
_{0\leq\alpha\leq1}\hstar_\alpha(\ftrue,\fbest)<e^{-j^2\epsilon
^2/4}$. In our approach, this corresponds to a stronger version of
Assumption 2b as stated in the theorem below. If $\sup_{j\geq1}
N_j<\infty$, then they show that $\E_0[\Pi(d(f,\fbest)\geq
\epsilon|Y_{1:n})]\rightarrow\ 0$. An analogous result using
our approach is as follows.
\begin{theorem}
\label{T:analogKV}
Let metric $d$ be such that $d$-balls in $\Fconv$ are convex sets.
Suppose Assumption 1 and a stronger version of Assumption 2b (where
$\delta$ depends on $\epsilon$) hold, i.e.,
\begin{itemize}
\item[] For any $\epsilon>0$, $U_{\epsilon}=\{f \in\Fconv:
d(f,\fbest) < \epsilon\} $ contains a set of the form $\{f\in\Fconv
: \inf_{0\leq\alpha\leq1} \hstar_{\alpha} (\ftrue,f) >
e^{-\epsilon^2}\}$.
\end{itemize}
Let $N_j$ be the minimum number of $d$-balls in $\Fconv$ of radius
$j\epsilon/3$ that cover the set $S_j=\{f\in\FF_0: j\epsilon\leq
d(f,\fbest) < (j+1) \epsilon\}$. If $\sup_{j\geq1}N_j < \infty$,
then
\[
\Pi(d(f,\fbest)>\epsilon|Y_{1:n})\rightarrow\ 0 \qquad\prob
_0\mbox{-a.s\xch{.}{..}}
\]
\begin{proof}
Along the lines of Lemma~\ref{lem2new}, we get
\begin{align*}
\prob_0 \left( \int_{U_{\epsilon}^c} \fn/\fbestn d\Pi >
e^{-n\beta} \right)
\leq e^{n\beta}\cdot \sum_{j=1}^{\infty} \inf_{0\leq\alpha\leq
1}\E_0\left( \int_{S_j}\fn/\fbestn d\Pi_j\right)^{\alpha}.
\end{align*}
Note that for any $f_1$ such that $d(f,f_1)\geq j\epsilon$, $B_{f_1}=\{
f: \ d(f,f_1)<j\epsilon/3\}$ does not intersect with $U_{j\epsilon
/2}$. Hence $\sup_{f\in B_{f_1}}\inf_{0\leq\alpha\leq1} \E_0
(f/\fbest)^{\alpha}<e^{-j^2\epsilon^2/4}$.

Using Proposition~\ref{prop.convexA} and the fact that $S_j$ can be
covered by $N_j$ convex sets of the form $B_{f_1}$, we get
\begin{align*}
\prob_0 &\left( \int_{U_{\epsilon}^c} \frac{\fn}{\fbestn} d\Pi
> e^{-n\beta} \right) \leq e^{n\beta}\cdot\sum_{j=1}^{\infty}
N_j\cdot e^{-nj^2\epsilon^2/4} \leq e^{n\beta}\cdot\sup_{j\geq
1}N_j \cdot\frac{e^{-n\epsilon^2/4}}{1-e^{-n\epsilon^2/4}}.
\end{align*}
A suitable choice of small enough $\beta$ and an application of
Borel--Cantelli lemma, gives that for $\beta_0<\beta$, $e^{n\beta
_0}\int_{U_{\epsilon}^c} \fn/\fbestn d\Pi\rightarrow 0 \quad  \prob
_0\mbox{-a.s\xch{.}{..}}$ This, along with Assumption 1, gives the result.
\end{proof}
\end{theorem}

As noted earlier, Assumptions 2a and 2b are stated in terms of $\Fconv
$, which makes them difficult to verify for non-convex models.
Assumption 2c helps handle the case of non-convex families. We now
derive consistency results under Assumption 2c and the following
continuity assumption.
\begin{itemize}
\item[] {\bf Assumption 4.} For any $f_1,f_2 \in\FF_0$ and for some
monotonically increasing function $\eta(\cdot)$ with $\eta(0)=0$ we have
\begin{eqnarray*}
\E_0\left| \frac{f_1}{\fbest}-\frac{f_2}{\fbest}\right|\leq\eta
(d(f_1,f_2)).
\end{eqnarray*}
\end{itemize}
Note that Assumption 4 is satisfied by $d=L_1(\mu_0)$, in which case
$\eta(\cdot)$ is just the identity function. If $\ftrue/\fbest\in
L_{\infty}(\mu)$ then the assumption is satisfied by $L_1(\mu)$.
Also if $\ftrue/\fbest\in L_2(\mu)$, an application of
Cauchy--Schwartz inequality shows that it is satisfied for $d=L_2(\mu)$.

Towards deriving consistency,
the next lemma shows that if Assumptions 2c and 4 hold then
$U_{\epsilon}^c$ can be covered by $d$-balls whose posterior
probabilities diminish to zero with increasing $n$.
\begin{lemma}
\label{lemnew4}
Let $U^{c}=\{f\in\FF_0: d(f,\fbest)\geq\epsilon\}$.
Suppose Assumptions 2c and 4 hold with respect to $U$ and a metric $d$.
Let $\alpha_0, \delta$ be as in Assumption 2c and $\eta(\cdot)$ as
in Assumption~4.
Then, for any $f_1\in U^{c} $ there exists an open ball $B(f_1, r)$
around $f_1$, with the radius $r$ depending only on $\delta, \alpha
_0$ and $\eta(\cdot),$ such that
\[
\E_0\left(\int_{B\left(f_1,r\right)}\fn/\fbestn d\Pi(f) \right
)^{\alpha_0}\leq e^{-n\frac{\delta}{2}}\cdot\Pi\left(B\left
(f_1,r\right)\right)^{\alpha_0}.
\]

\begin{proof} Let $r=\eta^{-1}(( e^{-\frac{\delta
}{2}}-e^{-\delta})^{\frac{1}{\alpha_0}})$ and $\nu
(\cdot)$ be any probability measure on $\FF_0$. \\Since $0<\alpha_0<1$,
\begin{eqnarray*}
&& \E_0\left(\int_{B\left(f_1,r\right)}f/\fbest\ d\nu(f) \right
)^{\alpha_0}\\
&& \leq\E_0\left(\int_{B\left(f_1,r\right)} \left|\frac
{f}{\fbest}-\frac{f_1}{\fbest}\right|d\nu(f) \right)^{\alpha_0}
+ \E_0\left(\int_{B\left(f_1,r\right)}\frac{f_1}{\fbest}d\nu(f)
\right)^{\alpha_0} \\
&& \mbox{ (then by Jensen's inequality)}\\
&& \leq\left(\int_{B\left(f_1,r\right)}\E_0\left|\frac
{f}{\fbest}-\frac{f_1}{\fbest}\right|d\nu(f) \right)^{\alpha_0}
+ \E_0\left(\frac{f_1}{\fbest} \right)^{\alpha_0}\nu(B\left
(f_1,r\right))^{\alpha_0}.
\end{eqnarray*}
By Assumption 4, the first term of the above inequality satisfies
\[
\left(\int_{B\left(f_1,r\right)}\E_0\left|\frac{f}{\fbest
}-\frac{f_1}{\fbest}\right|d\nu(f) \right)^{\alpha_0} \leq
(e^{-\frac{\delta}{2}} -e^{-\delta})\cdot\nu(B\left(f_1,r\right
))^{\alpha_0}.
\]
By Assumption 2c, the second term of the inequality is bounded as
\[
\E_0\left(\frac{f_1}{\fbest} \right)^{\alpha_0}\cdot\nu(B\left
(f_1,r\right))^{\alpha_0} < e^{-\delta}\cdot\nu(B\left
(f_1,r\right))^{\alpha_0}.
\]
Therefore, it follows that for any probability measure $\nu(\cdot)$
on $\FF_0$ we have
%
%e7 ###
\begin{eqnarray}
\E_0\left(\int_{B\left(f_1,r\right)}f/\fbest\ d\nu(f) \right
)^{\alpha_0}\leq e^{-\frac{\delta}{2}}\cdot\nu(B\left(f_1,r\right
))^{\alpha_0}. \label{ncindstep1}
\end{eqnarray}
Equation (\ref{ncindstep1}) is the result for $n=1$. An induction
argument on $n$ along the lines of Proposition~\ref{prop.convexA} can
now be used to obtain the result. To see this for $n=2$, note that, as
in the proof of Proposition~\ref{prop.convexA}, we can write
\[
\left[\frac{f_\nu^{(2)}(y_1,y_2)}{f^{\star(2)}(y_1,y_2)}\right
]^{\alpha_0} = \left[\frac{f_{\nu}(y_1|y_2)}{\fbest(y_1)}\right
]^{\alpha_0} \left[ \frac{f_{\nu}(y_2)}{\fbest(y_2)}\right
]^{\alpha_0}.
\]
Now, by (\ref{ncindstep1}), $\E_(\frac{f_{\nu}(y_2)}{\fbest
(y_2)})^{\alpha_0}\leq e^{-\frac{\delta}{2}}\cdot\nu(B
(f_1,r))^{\alpha_0}$. Further, since (\ref{ncindstep1}) holds
for any probability measure, taking the measure $\frac{f(y_2)}{f_{\nu
}(y_2)}d\nu$, we get
\[
\E_0 \left[\left(\frac{f_\nu(y_1|y_2)}{\fbest(y_1)}\right
)^{\alpha_0} \Big{|} y_2\right] \leq\ e^{-\frac{\delta}{2}}.
\]
It therefore follows that for any probability measure $\nu$ on $\FF_0$,
\[
\E_0 \left[\frac{f_\nu^{(2)}(y_1,y_2)}{f^{\star
(2)}(y_1,y_2)}\right]^{\alpha_0} \leq\ e^{-2\cdot\frac{\delta
}{2}} \cdot\nu(B\left(f_1,r\right))^{\alpha_0}.
\]
A similar induction argument for general $n$ completes the proof.
\end{proof}
\end{lemma}
To ensure that the total posterior probability of $U^c$ goes to zero,
we need to be able to cover it with sets of the form $B
(f_1,r)$ that satisfy a prior-summability assumption as in \cite
{DW2013} or metric entropy assumption as in \cite{Kleijn_van2006}. The
following theorem is an immediate consequence of
Lemma~\ref{lemnew4}.\vspace*{6pt}
\begin{theorem}
\label{DWKVequi}
Let $U^{c}=\{f\in\FF_0: d(f,\fbest)\geq\epsilon\}$. % and $d$ be a
%metric such that $d-$ balls in $\Fconv$ are convex.
Suppose Assumptions 1, 2c and 4 hold. Let $\alpha_0$ be as in
Assumption 2c and $\eta(\cdot)$ be as in Assumption 4. Suppose for
any given $r>0$, $r_{\alpha_0}=\eta^{-1}(r^{\frac{1}{\alpha_0}})$
and $\{B(f_j,r_{\alpha_0}), \ j\geq1\}$ is an open cover
of $U^{c}$ such that one of the following two conditions (a) or (b) holds:
\begin{itemize}
\item[] (a) $\sum_{j}\Pi( B(f_j, r_{\alpha_0})
)^{\alpha_0}<\infty$.
\item[] (b) Assumption 3 holds.
\end{itemize}
Then, $\Pi(U^{c} |Y_{1:n})\rightarrow 0 \quad \prob_0\mbox{-a.s\xch{.}{..}}$
\begin{proof} If condition (a) holds, then since $0<\alpha_0<1$, we get
\begin{align*}
\prob_0 \left( \int_{U^c} \fn/\fbestn d\Pi > e^{-n\beta} \right)
& \leq e^{n\beta}\sum_{j\geq1}\E_0\left( \int_{B(f_j, r_{\alpha
_0})} \fn/\fbestn d\Pi\right)^{\alpha_0}\\
& \leq e^{n\beta}\cdot e^{-n\delta/2}\cdot\sum_{j\geq1} \Pi
(B\left(f_1,r\right))^{\alpha_0}.
\end{align*}
A suitable $\beta$ and Borel--Cantelli Lemma give that for $\beta
_0<\beta$,
$e^{n\beta_0}\int_{U^c} \fn/\fbestn d\Pi\rightarrow\ 0\  \  \prob
_0\mbox{-a.s\xch{.}{..}}$ This, along with Proposition~\ref{post.denom}, gives
$\Pi(U^{c} |Y_{1:n})\rightarrow\ 0 \ \ \prob_0\mbox{-a.s\xch{.}{..}}$ \\
If condition (b) holds then the proof is along the same lines as for
Theorem~\ref{dconsis}.
\end{proof}
\end{theorem}

\begin{remark}
As noted earlier, Assumption 4 automatically holds for $d=L_1(\mu_0)$.
In that case, the function $\eta(\cdot)$ is just the identity
function, and the result based on condition (a) of Theorem~\ref
{DWKVequi} is analogous to Corollary 1 of \cite{DW2013}.
\end{remark}

\begin{remark}
\label{R:parametric}
Theorem~\ref{DWKVequi} can be easily applied to $i.i.d.$ parametric
models, i.e., when $\FF_0=\{f_{\theta}: \ \theta\in\Theta
\}$. Let $\fbest=f_{\theta^*}$ for some $\theta^*\in\Theta$
be the minimizer of Kullback--Leibler divergence from $\ftrue$. It is
easy to see that Assumption 1 is ensured as long as the prior $\Pi$
assigns a positive probability to every open $d$-neighborhood of
$\theta^*$ and $\E_0\log\frac{f_{\theta^{*}}}{f_{\theta}}$ is
continuous in $\theta$. Further, note that Theorems~\ref
{T:ass2c_compact} and \ref{assn2prime} in Section~\ref{S:L1weak}
provide sufficient conditions for Assumption 2c to hold with respect to
$L_1(\mu_0)$. However, to apply the results for the parametric model,
we need the assumption to hold with respect to the metric $d$. This can
be easily ensured if in addition to conditions of Theorems~\ref
{T:ass2c_compact} or \ref{assn2prime}, it can be verified that for
some monotonically increasing function $\zeta(\cdot)$, with $\zeta(0)=0$:
\[
\int|f_{\theta}-f_{\theta^*}|d\mu_0 \geq \zeta(d(\theta, \theta
^*)) ,\ \forall\ \theta\in\Theta.
\]
Then, by defining the metric on $\FF_0$ as $d(f_{\theta_1}, f_{\theta
_2}):= d(\theta_1,\theta_2)$, and using Assumption 4, Theorem~\ref
{DWKVequi} can be applied. We provide examples of this approach to
parametric models in Section~\ref{S:examples}.
\end{remark}
\begin{remark}
Assumption 4 is a continuity condition on $\E_0(\frac{f}{\fbest
})$. It is possible to work with an an alternative condition
that assumes continuity of $\E_0(\frac{f}{f_1})$ for any
$f_1 \in\FF_0$. In particular, if $f_1 \in U^{c}$, by Assumption 2c,
$\E_0( \frac{f_1}{\fbest})^{\alpha_0}<e^{-\delta}$.
Then, the conclusion analogous to that of Lemma~\ref{lemnew4}, which
is crucial for Theorem~\ref{DWKVequi}, can be obtained by defining an
open set $B_1:= \{\theta\in\Theta: \ E[\frac
{f_{1}}{\fbest} ]< e^{\frac{\delta}{2}} \}$. Then for
$\alpha=\frac{\alpha_0}{2}$ and any probability measure $\nu(\cdot
)$ on $B_1$, it can be shown (by using Cauchy--Schwartz and Jensen's
inequality) that
\begin{eqnarray*}
\E_0\left[ \left(\int_{B_1}\frac{f}{\fbest} d\nu\right)^{\alpha
}\right]&=&E\left[ \left(\frac{f_1}{\fbest} \right)^{\alpha
}\cdot\left(\int_{B_1}\frac{f}{f_{1}} d\nu\right)^{\alpha}\right
]\\
%&&\leq\left(E\left[ \left(\frac{f_1}{\fbest}\right)^{2\alpha}\right]
%\right)^{\frac{1}{2}} \cdot\left(E\left[ \left(\int_{B_1}
%\frac{f}{f_1} d\nu\right)^{2\alpha}\right]\right)^{\frac{1}{2}} \mbox{
%(by Cauchy-Schwarz inequality)}\\
&\leq& \left(\E_0\left[ \left(\frac{f_1}{\fbest}\right
)^{2\alpha}\right]\right)^{\frac{1}{2}} \cdot\left(\int_{B_1}\E
_0\left[ \frac{f}{f_1}\right]d\nu\right)^{\alpha} < e^{-\alpha
\frac{\delta}{2}}.
\end{eqnarray*}
Posterior consistency can then be derived if $U^c$ can be covered by
suitably many sets of the form $B_1$, e.g., when $\FF_0$ is compact.
We work with such an assumption in Section~\ref{S:inid} (Assumption D)
while extending results to the $i.n.i.d.$ case.
\end{remark}

%s4 ###
\section{ Weak and $L_1$ consistency}
\label{S:L1weak}
Assumptions 2a, 2b and 2c are crucial for Theorems~\ref{KL.compact},
\ref{dconsis} and \ref{DWKVequi}, respectively. These, we feel, are
in general hard to verify. Here, we focus on specific topologies and
discuss cases where these assumptions hold.

Recall $d\mu_0 = (f_0/ \fbest) \,d\mu$. Our interest is in two
topologies on $\Fconv$. First, the weak topology on $ L_1 (\mu_0)$
induced by $ L_\infty(\mu_0)$. The basic open neighborhoods of
$\fbest$ here are finite intersections of sets of the form
\[
\left\{ g \in L_1(\mu_0): |\int\phi_k g d\mu_0 - \int\phi\fbest
d\mu_0| < \epsilon_k , \qquad\phi_k\in L_\infty(\mu_0) \right\}.
\]
We will refer to this as the $\mu_0$-weak topology. The other topology
is the $ L_1$ topology which yields neighborhoods of the form $\{
g: \int|g - \fbest| d\mu_0 <\epsilon\}$. Of interest are
also the usual weak and total variation topologies on densities. These
correspond to $\mu$-weak topology and $L_1(\mu)$ topology. In the
context of consistency, our interest is in the concentration of the
posterior in neighborhoods of $\fbest$. We formally define
\begin{defn} The sequence of posterior distributions $\{ \Pi( \cdot|
Y_1, Y_2, \ldots, Y_n)\}_{n \geq1} $ is said to be $ \mu_0$-weakly
consistent if,
for any $\mu_0$-weak neighborhood $U$ of $\fbest$,
\[
\Pi(U | Y_{1:n}) \rightarrow1 \qquad\prob_0\mbox{-a.s\xch{.}{..}}
\]
\end{defn}
We will now show in the theorem below that when $\FF_0$ itself is
convex, Assumptions~2a, 2b and 2c are ensured with respect to weak and
$L_1(\mu_0)$ topologies.\vspace*{6pt}

\begin{theorem} \label{dinKV.convex}
If $\FF_0$ is convex then Assumptions 2a, 2b and 2c hold both with
respect to the $L_1(\mu_0)$ topology and the $\mu_0$-weak topology.
\begin{proof} First, using the Cauchy--Schwartz inequality, we get
\begin{align*}
\int| \fbest- f| d\mu_0 &= \int\left| 1- \frac{f}{\fbest}\right
| \ftrue d\mu= \int\left|\left(1-\sqrt{\frac{f}{\fbest}}\right
)\right|\cdot\left(1+\sqrt{\frac{f}{\fbest}}\right) \ftrue d\mu
\\
&\leq\left( \int\left|1-\sqrt{\frac{f}{\fbest}}\right|^2 \ftrue
d\mu\right)^{\frac{1}{2}}\cdot\left(\int \left(1+\sqrt{\frac
{f}{\fbest}}\right)^2 \ftrue d\mu\right)^{\frac{1}{2}}.
\end{align*}

Since $\FF_0$ is convex, by Lemma 2.3 of \cite{Kleijn_van2006}, $ \E
_0 \frac{f}{\fbest} \leq1$. Hence, the second term in the above
inequality is bounded because
\[
\E_0\left(1+\sqrt{\frac{f}{\fbest}}\right)^2 = 2\left(1+ \E
_0\frac{f}{\fbest}\right) \leq4.
\]
Similarly, for the first term,
\begin{align*}
\E_0\left(1-\sqrt{\frac{f}{\fbest}}\right)^2 &= \E_0\left
(1-\sqrt{\frac{f}{\fbest}}\right)^2 = 1+ \E_0\frac{f}{\fbest} -
2\E_0 \sqrt{\frac{f}{\fbest}} \\
& \leq2\left(1 - \E_0 \sqrt{\frac{f}{\fbest}}\right) = \frac
{1-\hstar_{\frac{1}{2}}(\ftrue,f)}{\frac{1}{2}} \leq K^*( \ftrue,f).
\end{align*}

Thus using Lemma~\ref{KVlem6}, we get
\[
\int| \fbest- f| d\mu_0 \leq2 \cdot\sqrt{\frac{1-\hstar_\alpha
(\ftrue,f)}{\alpha} }
\mbox{(with $\alpha =0.5$)} \leq2 \sqrt{ \Kstar(\ftrue,f)}.
\]
The last inequality ensures that Assumption 2a, 2b and 2c hold with
respect to $L_1(\mu_0)$. Since every weak neighborhood contains an
$L_1$-neighborhood, assumptions hold with respect to the $\mu_0$-weak
topology as well.
\end{proof}
\end{theorem}

\begin{remark}[{\bf$\mu_0$-Weak Consistency}]\label{weakconsis} By
Theorem~\ref{thm.convexA}, Assumption 2a along with Assumption 1
ensures $\mu_0$-weak consistency. This is because the complement of a
weak neighborhood is a finite union of convex sets. Further, by Theorem
\ref{dinKV.convex}, if $\FF_0$ is convex, Assumption 1 is enough to
ensure $\mu_0$-weak consistency.
\end{remark}

When $\FF_0$ is not convex, Assumptions 2a and 2b are not easy to
verify. In that case, it may be easier to work with Assumption 2c.
Next, we derive two results with sufficient conditions for Assumption
2c to hold with respect to the $L_1(\mu_0)$ metric. The first simpler
result below is obtained when $\FF_0$ is $L_1(\mu_0)$
compact.\vspace*{6pt}
\begin{theorem}
\label{T:ass2c_compact}
If $\FF_0$ is $L_1(\mu_0)$-compact, then Assumption 2c holds with
respect to $d=L_1(\mu_0)$.\vspace*{-6pt}
\begin{proof}
Suppose $f_1\in\FF_0$ is such that $d(f_1,\fbest)\geq\epsilon$.
Since $\fbest$ is assumed to be unique, $\exists\ \eta>0$ such that
$\Kstar(\ftrue,f_1)>\eta$ and further by Lemma~\ref{KVlem6},
$\exists\ 0<\alpha<1$ such that $\hstar_{\alpha}(\ftrue,
f_1)<1-\alpha\eta< e^{-\alpha\eta}$. Here $\alpha$ and $\eta$ may
depend on $f_1$. Now, let $B_{f_1}=\{f\in\FF_0: \ d(f,f_1)<r_{\alpha
}\}$, where $r_{\alpha}\leq(e^{-\alpha\eta/2}-e^{-\alpha\eta}
)^{\frac{1}{\alpha}}$. For $f\in B_{f_1}$, %and is chosen so that the
%open ball $B_{f_1}:=\{f: \ d(f, f_1)<r_\alpha\}$ does not intersect $
%\{f:d(f, \fbest)<\frac{\epsilon}{2}\}$.
since $0<\alpha<1$,
\begin{eqnarray*}
\hstar_{\alpha}(\ftrue,f) &\leq& \hstar_{\alpha}(\ftrue,f_1) +
\int\left|\left( \frac{f}{\fbest}\right)^{\alpha}- \left( \frac
{f_1}{\fbest}\right)^{\alpha}\right|\ftrue d\mu\\
&\leq& \hstar_{\alpha}(\ftrue,f_1) + \int\left| \frac{f}{\fbest
}- \frac{f_1}{\fbest}\right|^{\alpha}\ftrue d\mu\\
&\leq& \hstar_{\alpha}(\ftrue,f_1) + \left( \int\left|
f-f_1\right| d\mu_0 \right)^\alpha\leq e^{-\delta}, \mbox{ where
} \delta=\alpha\eta/2.
\end{eqnarray*}
The last step uses Jensen's inequality. We have essentially shown that
if $d(f_1,\fbest)\geq\epsilon$ there is an open $L_{1}(\mu_0)$-ball
$B_{f_1}$ around $f_1$ and $\exists$ $0<\alpha<1$, $\delta>0$ such
that $\hstar_{\alpha}(\ftrue,f)\leq e^{-\delta}$, for all $f\in
B_{f_1}$. Then, as noted in proof of Proposition~\ref{KVKLrelation} in
Appendix~\ref{proofs}, we would also have $\hstar_{\alpha'}(\ftrue
,f)<e^{-\delta}$ for all $\alpha'<\alpha$. Since $\FF_0$ is
$L_1(\mu_0)$-compact, $\{f: d(f,\fbest)\geq\epsilon\}$ can be
covered by finitely many such balls $B_{f_1}, B_{f_2}, \dots,
B_{f_k}$, thus obtaining an $\alpha$ and $\delta$ corresponding to
each ball. The result is obtained by choosing the minimum of these
finitely many $\alpha$'s and $\delta$'s and noting that $\{f\in\FF
_0: d(f,\fbest)\geq\epsilon\} \subseteq\cup_{i=1}^{k}B_{f_i}
\subseteq\{f\in\FF_0 : \hstar_{\alpha}(\ftrue, \fbest) \leq
e^{-\delta}\}$.\vspace*{-6pt}
\end{proof}
\end{theorem}
The following theorem and the corollary give sufficient conditions for
Assumption 2c to hold with respect to $L_1(\mu_0)$, when $\FF_0$ is
neither convex nor compact.\vspace*{6pt}
\begin{theorem}
\label{assn2prime}
If $\exists\ 0<\alpha_0<1$ such that $\sup_{f\in\FF_0}\E_0
(\frac{f}{\fbest})^{\alpha_0} \leq1$ and suppose\break $\sup
_{f\in\FF_0} \E_0(\frac{f}{\fbest})^2 < \infty$. Then Assumption
2c holds with respect to $d=L_1(\mu_0)$.\vspace*{-6pt}
\begin{proof}
Without loss of generality, assume $\alpha_0=\frac{1}{2^{K-1}}$ for
some $K>1$. Define $a:= (\frac{f}{\fbest})^{\frac
{1}{2^K}}$. Then
\begin{eqnarray*}
\E_0\left| \frac{f}{\fbest}-1\right| &=& \E_0\left| a^{2^k}-1
\right| =\E_0 \left[\left|a-1\right|\cdot\left|1+a^2+a^3+\cdots
+a^{2^k-1} \right|\right]\\
&\leq& \E_0 \left(\left|a-1\right|^2 \right)^\frac{1}{2}\cdot
\left( \E_0\left|1+a^2+a^3+\cdots+a^{2^k-1} \right|^2 \right
)^\frac{1}{2}.
\end{eqnarray*}\vadjust{\eject}
For the first term on the right-hand side of the above inequality, we have
\begin{eqnarray*}
&&\E_0\left| a-1\right|^2 = \E_0\left( \frac{f}{\fbest}\right
)^{\frac{1}{2^{k-1}}} +1 -2 \E_0\left( \frac{f}{\fbest}\right
)^{\frac{1}{2^{k}}}\leq2 \left(1-\E_0\left( \frac{f}{\fbest
}\right)^{\frac{1}{2^{k}}}\right).
\end{eqnarray*}
Note that every term within the expansion $\E_0|1+a^2+a^3+\cdots
+a^{2^k-1} |^2$ is of the form $\E_0a^l= \E_0( \frac
{f}{\fbest})^\frac{l}{2^{k}}$, where $l\leq2^{k+1}-2$. In
particular, the second term is bounded by some constant multiple of
$\sqrt{\sup_{f\in\FF_0} \E_0 (\frac{f}{\fbest}
)^2}$. Hence the result follows.
\end{proof}
\end{theorem}
\begin{corollary}
\label{corassn2prime}
If the log-likelihood ratio $\log\frac{f}{\fbest}$ is uniformly
bounded, then Assumption~2c holds with respect to $d=L_1(\mu_0)$.
\begin{proof}
By uniform boundedness, $\exists$ $\alpha$ that does not depend on
$f$ such that $\alpha\cdot| \log\frac{f}{\fbest}| <
\frac{1}{2}$. We note that when $t<1, e^t<\frac{1}{1-t}$. In
particular, for $t<1/2$, $e^t-1<t/(1-t)<2t$. Applying this inequality,
we get
\[
e^{\alpha\cdot \log\frac{f}{\fbest}} -1 < -2 \alpha\cdot\log
\frac{\fbest}{f}.
\]
Therefore, $\E_0(\frac{f}{\fbest})^{\alpha}< 1-2
\alpha\cdot\E_0\log\frac{\fbest}{f}$. Since $ 0\leq2 \alpha
\cdot\E_0\log\frac{\fbest}{f}<1$, we have $\E_0(\frac
{f}{\fbest})^{\alpha}\leq1$. Clearly, uniform boundedness
also ensures that $ \E_0(\frac{f}{\fbest})^2$ is uniformly bounded.
Hence Theorem~\ref{assn2prime} implies that Assumption 2c holds.
\end{proof}
\end{corollary}
\begin{remark}[{\bf$L_1(\mu_0)$-consistency}]
\label{L1consis}
Clearly, Theorem~\ref{dconsis} of the previous section can be applied
for $d=L_1(\mu_0)$, along with the sufficient conditions presented in
this section for verifying Assumptions 2b or 2c. In particular, we can
conclude by Theorems~\ref{DWKVequi} and \ref{T:ass2c_compact} that,
if $\FF_0$ is $L_1(\mu_0)$ compact, Assumption 1 is enough to ensure
$L_1(\mu_0)$-consistency. This is because Assumption 4 is
automatically satisfied by $L_1(\mu_0)$, the entropy condition will
hold by compactness and Assumption~2c holds due to compactness by
Theorem~\ref{T:ass2c_compact}.
\end{remark}

%s5 ###
\section{Examples}
\label{S:examples}
%s5.1 ###
\subsection{Mixture models}
The mixture models discussed in \cite{Kleijn_van2006} are covered by
our results. In particular, let $ y \mapsto f(y|z)$ be a fixed density
with respect to $\mu$ for each $z\in\mathcal{Z}$, and $ f(y,z) $ be
jointly measurable. For every probability measure $\nu$ on $\mathcal
{Z}$ let
\[
p_\nu(y) = \int f(y|z) d\nu(z).
\]
Let $\mathbf{M}$ be the set of probability measures on $\mathcal{Z}$.
Consider the model $\FF_0 = \{p_\nu: \nu\in\mathbf{M}\}$. Let
$\ftrue$ be the ``true\xch{''}{"} distribution and assume that $ \fbest\in\FF
_0$ satisfies $ K(\ftrue, \fbest) = \inf_{\nu\in\mathbf{M}}
K(\ftrue, p_\nu)$. As before set $ d\mu_0 = \frac{\ftrue}{\fbest}
d\mu$. Since $\FF_0$ is convex by Theorem~\ref{dinKV.convex},
Assumptions 2a, 2b, and 2c are satisfied. Therefore, as noted in Remark
\ref{weakconsis}, the posterior would be $\mu_0$-weakly consistent,
provided the prior satisfies Assumption 1.

\cite{Kleijn_van2006} specialize the above model to the case when $z
\mapsto f(y|z)$ is continuous for all $y$ and $\mathcal{Z}=[-M, M]$ is
compact. Under certain assumptions (including identifiability), they
show that there is a unique $\fbest$ that minimizes $K(\ftrue,f)$.
They further argue that $\FF_0$ is $L_1(\mu)$ compact. When $\{
f(y|z): z \in[-M, M]\}$ is the normal location family, they show that
their assumptions hold for the Dirichlet prior. Since this is a convex
family, by Theorem~\ref{dinKV.convex}, Assumptions 2a, 2b, and 2c hold
with respect to $L_1(\mu_0)$. If $ \ftrue/\fbest\in L_{\infty}(\mu
)$, then the map $T:(\FF_0,L_1(\mu))\mapsto(\FF_0,L_1(\mu_0))$,
defined by $T(f)=f$, is continuous. Therefore, $L_1(\mu)$-compactness
implies $L_1(\mu_0)$-compactness. Hence, Theorem~\ref{KL.compact}
implies that $L_1(\mu_0)$-consistency holds. Since $ \ftrue/\fbest
\in L_{\infty}(\mu)$, this also implies that $L_1(\mu)$-consistency holds.

%s5.2 ###
\subsection{Normal regression}
Consider the family of bivariate densities $\FF_0$ of the form
$f_{\theta}(y,x)=\phi(y-\theta(x))g(x)$ where $\phi(\cdot)$ is the
standard normal density and $\theta\in\Theta$, a class of uniformly
bounded continuous functions on the space of $X$. We assume that the
true density $\ftrue$ is such that $Y-\theta_0(X)\sim$ $p_0(\cdot
)$, a density with mean 0 that does not depend on $X$. It's easy to see
that $\fbest(y,x)=f_{\theta_0}(y,x)=\phi(y-\theta_0(x))g(x)$. We
are interested in posterior consistency with respect to the following metric:
\[
d(f_{\theta_1},f_{\theta_2})= \sqrt{ \E_0(\theta_1(X) -\theta
_2(X))^2 }.
\]
Let $Z=Y-\theta_0(X)$. We assume that $\E_0[e^{M|Z|}
]<\infty, \ \forall\ M>0$. For notational simplicity, we denote $\mu
_X:=\theta(X)-\theta_0(X)$. Note that
\begin{eqnarray*}
&& \log\frac{f_{\theta}}{f_{\theta_0}}= Z\cdot\mu_X - \frac{\mu
_X^2}{2},\\
&& \E_0\log\frac{f_{\theta}}{f_{\theta_0}} =- \E_0\frac{\mu
_X^2}{2} =- \E_0(\theta(X)-\theta_0(X))^2.
\end{eqnarray*}
This immediately ensures that Assumption 1 holds, as long as the prior
puts positive mass on $d$-neighborhoods of $\theta_0$. Towards
verifying Assumption 2c we note by using Taylor's approximation for
$\hstar_{\alpha}(\ftrue, f_{\theta})$ as a function of $\alpha$,
at $\alpha=0$, that for some $ \xi{\in(0,\alpha)}$,
\begin{eqnarray*}
&&\left|\E_0\left( \frac{f_{\theta}}{f_{\theta_0}}\right
)^{\alpha} -1 - \alpha\cdot\E_0\log\frac{f_{\theta}}{f_{\theta
_0}}\right| \leq\frac{\alpha^2}{2}\E_0\left[ \left(\log\frac
{f_{\theta}}{f_{\theta_0}}\right)^2 e^{\xi\log\frac{f_{\theta
}}{f_{\theta_0}}}\right].
\end{eqnarray*}
Since $\E_0[e^{M|Z|}]<\infty$ for any $M$ and $\mu_x=
\theta(x)-\theta_0(x)$ is uniformly bounded, the expectation on the
right-hand side of the above inequality will be bounded by some large
enough constant $C>0$.
Hence, we get
\begin{eqnarray*}
\left|\E_0\left( \frac{f_{\theta}}{f_{\theta_0}}\right)^{\alpha
} -1 + \alpha\cdot\E_0\frac{\mu_X^2}{2}\right| \leq\frac{\alpha^2}{2}C.
\end{eqnarray*}
Therefore,
\begin{eqnarray*}
\E_0\left( \frac{f_{\theta}}{f_{\theta_0}}\right)^{\alpha} \leq
1- \alpha\E_0(\theta(X)-\theta_0(X))^2 + \frac{\alpha^2}{2}C.
\end{eqnarray*}
For any $0<\epsilon<1$, note that if $d(\theta, \theta_0)= \sqrt{\E
_0(\theta(X)-\theta_0(X))^2}>\epsilon$, then
\begin{eqnarray}
\E_0\left( \frac{f_{\theta}}{f_{\theta_0}}\right)^{\alpha} &\leq
& 1- \alpha\epsilon^2 + \frac{\alpha^2}{2} C.\nonumber
\end{eqnarray}
In particular, with $ \alpha=\frac{\epsilon^2}{C}$, when $d(\theta,
\theta_0)>\epsilon$, we have
%
%e8 ###
\begin{eqnarray}
\E_0\left( \frac{f_{\theta}}{f_{\theta_0}}\right)^{\alpha} \leq
e^{-\frac{\epsilon^4}{2C}}. \label{ref1}
\end{eqnarray}
The above inequality ensures that Assumption 2c holds with respect to $d$.

Finally, to verify Assumption 4, first note that since $\mu_X$ is
uniformly bounded and $\E_0(e^{M|Z|})<\infty, \ \forall\ M$, for
some $C_1>0, M_1>0$, we have
\begin{eqnarray*}
&&\E_0\left|\frac{f_{\theta_1}}{f_{\theta_0}}-\frac{f_{\theta
_2}}{f_{\theta_0}}\right| = \E_0 \left[\left|\frac{f_{\theta
_1}}{f_{\theta_2}}-1\right|\cdot\frac{f_{\theta_2}}{f_{\theta
_0}}\right]\leq C_1\cdot\E_0 \left[\left|\frac{f_{\theta
_1}}{f_{\theta_2}}-1\right|\cdot e^{M_1|Z|}\right].
\end{eqnarray*}
Now, denote $\mu'(X)=\theta_1(X)-\theta_2(X)$. Again, using Taylor's
formula and the fact that $\mu'_X$ is uniformly bounded, we get that
for some $C_2>0, M_2>0$,
\begin{eqnarray*}
\left|\left( \frac{f_{\theta_1}}{f_{\theta_2}}\right)-1 \right|
&=& \left|e^{\left(Z\cdot\mu'_X-\frac{\mu'^2_X}{2}\right)}-1
\right| \\ &\leq& \sup_{0<\xi<1} \left[|\mu'_X|\cdot\left
|Z-\frac{\mu'_X}{2}\right| e^{\xi.\left(Z\cdot\mu'_X-\frac{\mu
'^2_X}{2}\right)}\right]\\ & \leq& C_2\cdot\left[|\mu'_X|\cdot
\left|Z-\frac{\mu'_X}{2}\right| e^{M_2|Z|}\right].
\end{eqnarray*}
Therefore, putting the above two inequalities together, we get for some
$M>0,\  C>0$,
\begin{eqnarray*}
\E_0\left|\frac{f_{\theta_1}}{f_{\theta_0}}-\frac{f_{\theta
_2}}{f_{\theta_0}}\right|&\leq& C \cdot\left[|\mu'_X|\cdot\left
|Z-\frac{\mu'_X}{2}\right| e^{M|Z|}\right]\\
&\leq& C\cdot\left(\E_0\left[|\mu'^2_X \right]\right)^{\frac
{1}{2}}\cdot\left( \E_0\left[\left|Z-\frac{\mu'_X}{2}\right|^2
e^{2M|Z|}\right]\right)^{\frac{1}{2}}.
\end{eqnarray*}
The last step uses Cauchy--Schwartz inequality. Since the last term in
the above inequality is finite, for some suitably large $K>0$, we can write
\begin{eqnarray*}
\E_0\left|\frac{f_{\theta_1}}{f_{\theta_0}}-\frac{f_{\theta
_2}}{f_{\theta_0}}\right|\leq(\E_0(\theta_1(X)-\theta
_2(X))^2)^{\frac{1}{2}} \cdot K,
\end{eqnarray*}
which implies that Assumption 4 holds. Hence Theorem~\ref{DWKVequi} is
applicable, as long as the prior-summability (part (a)) or the entropy
condition (part (b)) of the theorem holds.

%s5.3 ###
\subsection{Bayesian quantile regression}
\label{egBQRiid}
Consider the family of bivariate densities $\FF_0$ of the form
$f(y,x)=\phi(y-\theta(x))g(x)$ where $\phi(\cdot)$ is the
asymmetric Laplace density given by $\phi(z)=\tau(1-\tau)e^{-z(\tau
-I_{(z\leq0)})}, \ z\in(-\infty, \infty)$ with $I_{(\cdot)}$ being
the indicator function, $0<\tau<1$ and $\theta\in\Theta$, a class
of uniformly bounded continuous functions on the space of $X$. It is
easy to check that the $\tau${th} quantile of $\phi$ is 0. Hence,
this is one particular formulation used for Bayesian quantile
regression (see \citealt{Yu_moyeed2001}). We assume that the true
density is such that $Y-\theta_0(X)\sim$ $p_0(\cdot)$, a density
which does not depend on $X$ and whose $\tau${th} quantile is 0. It's
easy to see that $\fbest(y,x)=f_{\theta_0}(y,x)=\phi(y-\theta
_0(x))g(x)$ (see Proposition 1 in \cite{sriram.rvr.ghosh.2013}). We
are interested in posterior consistency with respect to the following metric:
\[
d(f_{\theta_1},f_{\theta_2})= \E_0\left|\theta_1(X) -\theta
_2(X)\right|.
\]
Let $Z=Y-\theta_0(X)$. It can be seen that (see Lemma 1 of \citealt
{sriram.rvr.ghosh.2013})
%e9 ###
\begin{eqnarray}
\left|\log\frac{f_{\theta_1}}{f_{\theta_2}}\right|&\leq& \left
|\theta_1(X)-\theta_2(X) \right|, \label{logLbound}\\
\E_0\left|\log\frac{f_{\theta_1}}{f_{\theta_2}}\right| &\leq&
\E_0\left|\theta_1(X)-\theta_2(X) \right|.\nonumber
\end{eqnarray}
This immediately ensures that Assumption 1 holds, as long as the prior
puts positive mass on $d$-neighborhoods of $\theta_0$. Further, since
$\theta$ are uniformly bounded, the first of the above two
inequalities ensures that $\frac{f_{\theta_1}}{f_{\theta_2}}$ is
uniformly bounded. By Corollary~\ref{corassn2prime} to Theorem~\ref
{assn2prime}, it follows that Assumption 2c will be satisfied with
respect to $L_1(\mu_0)$. Further, we argue that Assumption 2c holds
with respect to the metric $d$. To see this, first, it can be checked
using the form of asymmetric Laplace density that
\[
\left|\frac{f_{\theta}}{f_{\theta_0}} -1 \right| \geq
\begin{cases}\left(1-e^{-(\theta-\theta_0)(1-\tau)}\right)\cdot
I_{(Z\leq0)} \ \text{if} \ \theta-\theta_0\geq0,\\
\left(1-e^{(\theta-\theta_0) \tau}\right)\cdot I_{(Z> 0)} \ \ \ \
\ \ \  \text{if} \ \theta-\theta_0< 0.
\end{cases}
\]
Since $|\theta(X)-\theta_0(X)|$ is assumed to be uniformly bounded,
we can further say that there exists a constant $C_0>0$ such that
\[
\left|\frac{f_{\theta}}{f_{\theta_0}} -1 \right| \geq C_0 |\theta
(X)-\theta_0(X)|\cdot\left(I_{Z\leq0} \cdot I_{\theta-\theta
_0\geq0} + I_{Z> 0} \cdot I_{\theta-\theta_0<0}\right).
\]
Now, noting that $\E_0[I_{Z\leq0} | X]=\prob_0(Z\leq0
|X) =\tau$, we get
\begin{eqnarray*}
\E_0\left|\frac{f_{\theta}}{f_{\theta_0}} -1 \right| &\geq& C_0
\E_0 \left[|\theta(X)-\theta_0(X)|\cdot\left( \tau\cdot
I_{\theta-\theta_0\geq0} + (1-\tau)\cdot I_{\theta-\theta
_0<0}\right)\right]\\
&\geq& C_0 \min(\tau, 1-\tau)\cdot\E_0|\theta(X)-\theta_0(X)|\\
&=& C_0 \min(\tau, 1-\tau)\cdot d(f_{\theta}, f_{\theta_0}).
\end{eqnarray*}
Since we have already argued that Assumption 2c holds for $L_1(\mu
_0)$, the above inequality ensures that it also holds with respect to $d$.
Finally, to check Assumption 4, we use \eqref{logLbound} and the fact
that $\frac{ f_{\theta_2(X)}}{f_{\theta_0(X)}}$ is uniformly
bounded, to get that for some $C_1>0$,
\begin{eqnarray*}
&&\E_0\left| \frac{f_{\theta_1(X)}}{f_{\theta_0(X)}}-\frac
{f_{\theta_2(X)}}{f_{\theta_0(X)}} \right| = \E_0\left[\left|
\frac{f_{\theta_1(X)}}{f_{\theta_2(X)}}-1 \right|\cdot\frac{
f_{\theta_2(X)}}{f_{\theta_0(X)}}\right]\leq C_1 \cdot\E_0\left|
\frac{f_{\theta_1(X)}}{f_{\theta_2(X)}}-1 \right|.\\
&&\mbox{Hence, }\\
&&\E_0\left| \frac{f_{\theta_1(X)}}{f_{\theta_2(X)}}-1 \right
|\leq\max\left(\E_0\left|e^{-\left|\theta_1(X)-\theta
_2(X)\right|}-1 \right|, \E_0\left|e^{\left|\theta_1(X)-\theta
_2(X)\right|}-1 \right| \right),\\
&& \mbox{which by Taylor's formula is}\\
&&\leq C \cdot\E_0\left|\theta_1(X)-\theta_2(X)\right| \mbox{
for some constant C}.
\end{eqnarray*}
Therefore, Assumption 4 holds. Hence Theorem~\ref{DWKVequi} is
applicable, as long as the prior-summability (part (a)) or the entropy
condition (part (b)) of the theorem holds.
%s6 ###
\section{An extension to $i.n.i.d.$ models}
\label{S:inid}
The ideas developed in the previous sections allow us to begin some
easy and direct applications to $i.i.d.$ parametric models and to the
case of independent but non-identically distributed ($i.n.i.d.$)
response. In this section, we outline these ideas. In the interest of
flow, the results are presented here and the proofs are deferred to
Appendix~\ref{inidproofs}.

We will assume that the distribution of the response $Y$ is determined
in principle by the knowledge of a covariate vector ${\bf X}$. In other
words, there exists an unknown ``true\xch{''}{"} density function $f_{0 {\bf
x}}(\cdot)$ with ${\bf x}\in\mathcal{X}$, such that $Y|{\bf X}={\bf
x}\sim f_{0{\bf x}}$. So, for the $i${th} observed response $Y_i$ with
covariate value ${\bf X}_i={\bf x}_i$, $Y_i \sim f_{0 {\bf x}_i}$. The
${\bf X}_i$ could be non-random and hence $Y_1, Y_2,\dots, Y_n$ are
independent but non-identically distributed. $ \E_{{\bf x}}[\cdot]$
will denote the expectation w.r.t. the density $f_{0{\bf x}}$. We will
denote by $\prob_{x}$, the probability with respect to $f_{0x}$ and
$\prob_0$ with respect to the infinite product measure $f_{0{\bf
x}_1}\times f_{0{\bf x}_2}\times\cdots,$ and by $\E_0[\cdot]$ the
expectation w.r.t. this product measure.

Suppose we have a family of densities $\FF_0=\{f_{t}: \ t\in
[-M,M] \}$. Let $\Theta$ be a class of continuous functions
from $\mathcal{X}$ to $[-M,M]$. For ease of notation, we write $\theta
({\bf x})$ as $\theta_{\bf x}$. The specified model is that $Y_i \sim
f_{\theta_{{\bf x}_i}}$, where $\theta\in\Theta$ is the unknown
possibly infinite dimensional parameter. First, we make the following
assumption with regards to the covariates and the parameter space.
\begin{itemize}
\item[] {\bf Assumption A.} The covariate space $\mathcal{X}$ is
compact w.r.t. a norm $\|\cdot\|$ and $\Theta$ is a compact subset of
continuous functions from $\mathcal{X}\rightarrow\mathcal{R}$
endowed with the sup-norm metric, i.e., $d(\theta_1, \theta_2)=\sup
_{x\in\mathcal{X}}|\theta_1(x) -\theta_2(x)|$.
\end{itemize}

A straightforward parametric example for $\Theta$ would be any class
of smooth functions defined on a compact set $\mathcal{X}$ and
parametrized by finitely many parameters, also taking values on some
compact set. In this case, sup-norm metric would be equivalent to the
Euclidean metric on the finite dimensional parameter space. As an
example for a non-parametric class of functions, let $\mathcal
{X}=[0,1]$ and let $\Theta$ be the sup-norm closure of the collection
of polynomials $\mathcal{S}$ defined on $[0,1]$ given by $\mathcal
{S}:=\{\theta\ : \ \theta(x)= \sum_{j=1}^{k} a_j x^j \ \mbox{ for
some } k\geq1 \mbox{ and such that } a_j\leq\frac{1}{j^3} \}$. It
can be checked that this class is equi-uniformly-continuous and
uniformly bounded. Hence, $\Theta$ will be compact by Arzel\`a--Ascoli theorem.

Let $\Pi(\cdot)$ be a prior on the parameter space $\Theta$ and let
$\theta^*$ be the minimizer (with respect to $\theta$) of $\E_{{\bf
x}}\log\frac{f_{0 {\bf x}}}{f_{\theta_{\bf x}}}$ for all ${\bf x}$.
As before, we can write the posterior probability of a set $U^{c}= \{
\theta\in\Theta: d(\theta, \theta^*)>\epsilon\}$ as follows:
\begin{eqnarray*}
\Pi(U^{c}|Y_{1:n}):=\frac{\int_{U^{c}} \prod_{i=1}^{n}\frac
{f_{\theta_{{\bf x}_i}}(Y_i)}{f_{\theta^*_{{\bf x}_i}}(Y_i)} d\Pi
(\theta)}{\int_{\Theta} \prod_{i=1}^{n}\frac{f_{\theta_{{\bf
x}_i}}(Y_i)}{f_{\theta^*_{{\bf x}_i}}(Y_i)} d\Pi(\theta)}=:\frac
{R'_{1n}}{R'_{2n}}.
\end{eqnarray*}
We will be interested in probabilities of complements of sup-norm
neighborhoods of $\theta^*$. The following useful lemma shows that if
the functions $\theta$ and $\theta^*$ differ at a point ${\bf x}_0$,
then they will necessarily differ on a neighborhood around ${\bf x}_0$
as well.
\begin{lemma}
\label{lemnhbd}
Suppose Assumption A holds. Let $\theta' \in U^c$ and ${\bf x}_0 \in
\mathcal{X}$ be such that $|\theta'_{{\bf x}_0}-\theta^*_{{\bf
x}_0}|>\epsilon$. Then $\exists\ \delta'$ such that $\forall\ {\bf
x} \ : \ \|{\bf x}-{\bf x}_0\|<\delta'$ we have $|\theta'_{\bf
x}-\theta^*_{\bf x}| \geq\frac{\epsilon}{2}$.
\end{lemma}
Clearly, for the posterior probability of $\sup|\theta({\bf x})
-\theta^*({\bf x})|>\epsilon$ to go to 0, we need that $|\theta({\bf
x}_i) -\theta^*({\bf x}_i)>\epsilon$ for infinitely many of the
design values ${\bf x}_i$. The following assumption is a way to
formalize this notion.
\begin{itemize}
\item[] {\bf Assumption B.} For any given ${\bf x}_0\in\mathcal{X}$,
$\delta'>0$, let $A_{{\bf x}_0,\delta'}=\{{\bf x}: \ \|{\bf x}-{\bf
x}_0\|<\delta' \}$ and $I_{A_{{\bf x}_0,\delta'}}({\bf x})$ be the
indicator function which is $1$ when ${\bf x}\in A_{{\bf x}_0}$ and $0$
otherwise. Then, $\kappa({\bf x}_0, \delta')= \liminf_{n\geq1}
\frac{1}{n}\sum_{i=1}^{n}I_{A_{{\bf x}_0,\delta'}}({\bf x}_i)>0.$
\end{itemize}

Such a condition can be seen to hold in various situations. A simple
instance would be when $\mathcal{X}$ is a finite set $\{a_1, a_2,
\dots, a_k \}$ and when the covariates $X_i$ take each of these values
$a_j$ a fixed proportion of times. Another example would be when the
design set $\{x_i, \ i\geq1\} $ is dense in a closed interval say $[0,
1]$ such that the proportion of $x_i$'s falling in any sub-interval is
proportional to the interval length or a fixed function of the interval
length. More generally, it could also be designs where the proportion
of $x_i$ could vary, for example, twice as many $x_i$'s samples on $[0,
0.5]$ versus $[0.5, 1],$ etc.

Towards deriving conditions for $\Pi(U^{c}|Y_{1:n})\rightarrow\ 0$,
we first note that the next two assumptions are equivalent to
Assumption 1 of the $i.i.d.$ case and help control the denominator of
the posterior probability, as seen in Proposition~\ref{denominid} below.
\begin{itemize}
\item[] {\bf Assumption C.} $\exists$ $\theta^* \in\Theta$ such
that $\theta^*_{{\bf x}}=\arg\min_{t\in[-M,M]} \E_{{\bf x}}\log
\frac{f_{0 {\bf x}}}{f_{t}}, \ \forall\ {\bf x} \ \in\mathcal{X}$
and $\theta^*$ is in the sup-norm support of $\Pi$.
\item[] {\bf Assumption D.} $\E_{{\bf x}}[ \log\frac
{f_{t}}{f_{t'}}]$ and and $\E_{{\bf x}}[ (\frac
{f_{t}}{f_{t'}})^{\alpha}]$ for every $\alpha\in[0,1]$
are continuous functions in $({\bf x},t,t')\in\mathcal{X}\times[-M,
M]^2$ and $\E_{\bf x}\log^2\frac{f_t}{f_{t'}}$ is uniformly bounded
for $({\bf x},t,t')\in\mathcal{X}\times[-M, M]^2$.
\end{itemize}

This continuity assumption is in the same spirit as Assumption 4 in the
$i.i.d.$ case. Such an assumption will hold if $\frac
{f_t(y)}{f_t'(y)}$ is continuous in $(t,t')$ for each $y$ and if the
true density $f_{0{\bf x}}(y)$ is continuous in $x$ for each $y$, and
can be bounded by an integrable function in $y$. The boundedness
condition on the second moment of the log-likelihood ratio is to enable
the application of the strong law of large numbers (SLLN) for
independent random variables and is used in the proof of Proposition
\ref{denominid}.
\begin{proposition}
\label{denominid}
Suppose Assumptions A, C and D hold, then
\[
\mbox{ for any } \beta>0, \  e^{n\beta}R'_{2n} \rightarrow\ \infty
\qquad\prob_0\mbox{-a.s\xch{.}{..}}
\]
\end{proposition}
The next assumption helps relate the sup-norm metric with the
Kullback--Leibler divergence and is analogous to Assumption 2a.
\begin{itemize}
\item[] {\bf Assumption E.} For any $\epsilon>0$, $\exists\ \delta
\in(0,1)$ and $\alpha_0 \in(0,1)$ such that
\[
\left\{ t \in[-M, M]: \E_{{\bf x}}\log\frac{f_{\theta^*_{\bf
x}}}{f_{t}}<\delta\right\} \subseteq\left\{ t: \ |t-\theta^*_{\bf
x}|<\epsilon\right\}, \ \forall\ {\bf x}\in\mathcal{X}.
\]
\end{itemize}

Without delving into the details, we just note here that sufficient
conditions for this assumption to hold can be derived based on ideas
developed in Section~\ref{S:L1weak}. For example, as in Corollary~\ref
{corassn2prime}, a sufficient condition would be that $\log\frac
{f_{\theta(x)}}{f_{\theta^*_x}}$ is uniformly bounded.

The next lemma and proposition help obtain a result that is analogous
to Lemma~\ref{lemnew4}.
\begin{lemma}
\label{sepcondinid}
Let $U^{c}=\{\theta: \ \sup_{{\bf x} \in\mathcal{X}}|\theta
({\bf x})-\theta^*({\bf x})|>\epsilon\}$. If Assumptions A to
E hold, then $\exists$ $\delta_1\in(0,1) $ such that for every
$\theta'\in U^{c}$, an $ \alpha'\in(0,1)$ can be chosen such that
\[
\E_0 \left(\prod_{i=1}^{n}\frac{f_{\theta'_{{\bf
x}_i}}(Y_i)}{f_{\theta^*_{{\bf x}_i}}(Y_i)} \right)^{\alpha'} <
e^{-n \delta_1}\ \mbox{ for all sufficiently large } n.
\]
\end{lemma}
The next proposition is analogous to Lemma~\ref{lemnew4} and helps
control the numerator of the posterior probability.
\begin{proposition}
\label{numerinid}
Suppose Assumptions A to E hold. Then for any $\theta'\in U^{c}, \
\exists$ an open set $A_{\theta'}$ containing $\theta'$ such that
for some $\alpha\in(0,1)$, $\delta\in(0,1)$ and for any probability
measure $\nu(\cdot)$ on $A_{\theta'}$, for all sufficiently large
$n$, we have
\begin{eqnarray*}
\E_0\left[ \left(\int_{A_{\theta'}}\prod_{i=1}^{n}\frac
{f_{\theta_{{\bf x}_i}}(Y_i)}{f_{\theta^*_{{\bf x}_i}}(Y_i)} d\nu
(\theta) \right)^{\alpha}\right]< e^{-n\alpha\frac{\delta}{2}}.
\end{eqnarray*}
\end{proposition}
Finally, we obtain the result for the $i.n.i.d.$ case.
\begin{theorem}
\label{thminid}
Suppose that Assumptions A to E hold. Then,
\[
\Pi(U^c|Y_{1:n})\rightarrow\ 0 \qquad\prob_0\mbox{-a.s\xch{.}{..}}
\]
\begin{proof} Note that Proposition~\ref{numerinid} can be applied by
taking $ \nu(\cdot)=\frac{\Pi(\cdot)}{\Pi(A_{\theta^\prime})}$.
So, for any $\theta^{\prime}\in U^c$, $\exists\ \delta>0$ and an
open set $A_{\theta^\prime}$ containing $\theta^\prime$ such that,
for sufficiently large $n$,
\begin{eqnarray*}
&&\prob_0\left( \left(e^{n\frac{ \delta}{4}}\int_{A_{\theta
^\prime}} \prod_{i=1}^{n}\frac{f_{\theta_{{\bf
x}_i}}(Y_i)}{f_{\theta^*_{{\bf x}_i}}(Y_i)}d\Pi(\theta)\right
)^{\alpha}> \epsilon^{\alpha} \right)\\
&&\leq\frac{\Pi^{\alpha}(A_{\theta^\prime})}{\epsilon^{\alpha
}}\cdot\E_0\left(e^{n\frac{ \delta}{4}}\int_{A_{\theta^\prime}}
\prod_{i=1}^{n}\frac{f_{\theta_{{\bf x}_i}}(Y_i)}{f_{\theta^*_{{\bf
x}_i}}(Y_i)}\frac{d\Pi(\theta)}{\Pi(A_{\theta^\prime})}\right
)^{\alpha}\leq\frac{e^{-n\alpha\frac{\delta}{4}}}{\epsilon
^{\alpha}}.
\end{eqnarray*}
Therefore, by Borel--Cantelli lemma, we can conclude that
\[
e^{n\frac{ \delta}{4}}\int_{A_{\theta^{\prime}}} \prod
_{i=1}^{n}\frac{f_{\theta_{{\bf x}_i}}(Y_i)}{f_{\theta^*_{{\bf
x}_i}}(Y_i)}d\Pi(\theta) \rightarrow\ 0 \qquad\prob_0\mbox
{-a.s\xch{.}{..}}
\]
By Proposition~\ref{denominid}, it follows in particular that
\[
e^{n\frac{ \delta}{4}}\int_{\Theta} \prod_{i=1}^{n}\frac
{f_{\theta_{{\bf x}_i}}(Y_i)}{f_{\theta^*_{{\bf x}_i}}(Y_i)}d\Pi
(\theta) \rightarrow\ \infty \qquad\prob_0\mbox{-a.s\xch{.}{..}}
\]
Considering the ratio of the above two quantities immediately gives
$\Pi(A_{\theta^\prime}|Y_{1:n})\rightarrow\ 0 \ \ \prob_0\mbox{-a.s.}$\xch{}{.}
By compactness, $U^c$ can be covered by finitely many sets of the form
$A_{\theta^\prime}$. Hence the result follows.
\end{proof}

\end{theorem}
The proofs of lemmas and propositions discussed in this section are
provided in Appendix~\ref{inidproofs}. As an application, we briefly
discuss an example based on Bayesian quantile regression with
$i.n.i.d.$ response.
%s6.1 ###
\subsection{Example: Bayesian nonlinear quantile regression\xch{}{.}}
\label{eg2}
Similar to Section~\ref{egBQRiid}, consider a family of densities $\FF
_0 =\{f_{t}: t \in[-M, M] \}$, where $f_t(y)=\phi(y-\theta(x))g(x)$
where $\phi(\cdot)$ is the asymmetric Laplace density given by $\phi
(z)=\tau(1-\tau)e^{-z(\tau-I_{(z\leq0)})}, \ z\in(-\infty, \infty
)$ with $I_{(\cdot)}$ being the indicator function, $0<\tau<1$. Let
the ``true\xch{''}{"} quantile function of $Y$ given covariate ${\bf X}$ be
$\theta_0({\bf X})$.

By the properties of ALD, it can be observed (see Proposition 1, Lemmas
1 and 2 of \citealt{sriram.rvr.ghosh.2013}) that (a) $\theta^*=\theta
_0$, (b) $|\log\frac{f_{t^{\prime}}}{f_{t}}|\leq|t-t^\prime|,$ and
(c) that if $|t-\theta^*_{{\bf x}_i}|> \epsilon$ then
\[
\E_{{\bf x}_i}\log\frac{f_{\theta^*_{{\bf x}_i}}}{f_{\theta}}>
\delta_{{\bf x}_i} = \frac{\epsilon}{2}\cdot\min\left\{P_{0{\bf
x}_i}\left(0<Y_i-\theta^*_{{\bf x}_i}<\frac{\epsilon}{2}\right),
P_{0{\bf x}_i}\left(-\frac{\epsilon}{2}<Y_i-\theta^*_{{\bf
x}_i}<0\right) \right\}.
\]

As discussed above, there are various examples where Assumptions A and
B would hold. One may consider any of those possibilities for the
current example as well.

If we consider any prior that puts positive mass on sup-norm
neighborhoods of $\theta_0$, that along with observation (a) ensures
that Assumption C holds.

It follows by Assumption A and observation (b) above that $\frac
{f_{t}}{f_{t^\prime}}$ is uniformly bounded and by property of ALD
that it is continuous in $(t,t^\prime)$. If we assume that the true
density function $f_{0{\bf x}}(y)$ is continuous in $x$ for each $y$
and can be dominated by an integrable function in $y$, then an
application of dominated convergence theorem (DCT) would ensure that
Assumption D holds.

Finally, using observation (c), if $\prob_{0{\bf x}}(0<Y-\theta
^*_{{\bf x}}<\frac{\epsilon}{2})$ and $\prob_{0{\bf x}}
(-\frac{\epsilon}{2}<Y-\theta^*_{{\bf x}}<0) $ (where $Y\sim
\prob_{0{\bf x}}$) are continuous and positive functions of ${\bf x}$,
then $\{\delta_{{\bf x}_i},\ i\geq1\}$ can be uniformly bounded below
by a positive number. Hence, Assumption E would be satisfied.

\appendix
%\renewcommand\thesection{Appendix}
%s7 ###
\section{Supporting results and proofs}
\label{proofs}
We now state some technical results used in the paper.
The first useful result given below is same as Lemma 6.3 of \cite
{Kleijn_van2006}.
\begin{lemma}
\label{KVlem6}
As $\alpha\downarrow0$, $\frac{1-\hstar_\alpha(\ftrue,f)}{\alpha
} \uparrow\Kstar(\ftrue, f)$.
\end{lemma}
Proposition~\ref{KVKLrelation} in Section~\ref{S:prelim} shows that
the three notions of divergence we consider in this paper are
equivalent for convex sets. The proof of this result proceeds via the
minimax theorem. We state below the minimax theorem due to \cite{Sion}
Corollary~3.3. and relevant lemmas leading up to the proof of
Proposition~\ref{KVKLrelation}.

Let $M,N$ be convex sets. A function $g$ on $ M \times N$ is {\it
quasi-concave} in $M$ if $\{\nu: g(\mu,\nu) \geq c\}$ is a convex
set for any $\mu\in M$ and real $c$. Likewise, $g$ is {\it
quasi-convex} in $N$ if $\{\mu: g(\mu,\nu) \leq c\}$ is a convex set
for any $\nu\in N$ and real $c$. The function $g$ is {\it
quasi-concave--convex} if it is quasi-concave in $M$ and quasi-convex
in $N$. Similarly, if $g(\cdot, \nu)$ is upper semi-continuous (usc)
for any $\nu\in N$ and if $g(\mu, \cdot)$ is lower semi-continuous
(lsc) for any $\mu\in M$, then it is said to be {\it usc--lsc}. Then
\citet{Sion} proved the following.
\begin{theorem}
\label{minimax}
Let $M$ and $N$ be convex sets one of which is compact, and $g$ a
quasi-concave--convex and usc--lsc function on $ M \times N$. Then
\[
\sup_{\mu\in M} \inf_{\nu\in N} g(\mu,\nu) = \inf_{\nu\in N}
\sup_{\mu\in M} g(\mu,\nu).
\]
\end{theorem}
The next lemma investigates the relevant properties needed on the
function $\hstar_\alpha(\ftrue, f)$ so as to apply the minimax
theorem. For clarity, we recall the definition of $\hstar_{\alpha}$
and note that
\[
\hstar_{\alpha}(\ftrue,f)=
\begin{cases}
\E_0 \left(\frac{f}{\fbest}\right)^\alpha&\text{ if } 0 < \alpha
< 1, \\
1 &\text{ if } \alpha=0, \\
\E_0 \left(\frac{f}{\fbest}\right) &\text{ if } \alpha=1.
\end{cases}
\]
\begin{lemma}
\label{lemma_h}
The function $\hstar_\alpha(\ftrue, f)$ is concave in $f$ and convex
in $\alpha$. Further, for fixed $\alpha$, it is continuous in $f$ in
the $L_1(\mu_0)$-topology, where $d\mu_0 = (\ftrue/ \fbest) \,d\mu
$. Also, for fixed $f$, $\hstar_\alpha(\ftrue, f)$ is continuous in
$\alpha$.
\end{lemma}

\begin{proof}
Concavity and convexity are easy to check. Continuity in $f$ follows by
noting that, for $\phi\in\Fconv$,
\begin{align*}
|\hstar_\alpha(\ftrue, \phi) - \hstar_\alpha(\ftrue, f)| & =
\Bigl| \int\Bigl(\frac{\phi}{\fbest} \Bigr)^\alpha f_0 d\mu-
\int\Bigl( \frac{f}{\fbest} \Bigr)^\alpha f_0 d\mu\Bigr| \\
& \leq\int\Bigl| \frac{\phi}{\fbest} - \frac{f}{\fbest} \Bigr
|^\alpha f_0 d\mu\\
& \leq\Bigl( \int| \phi-f| d\mu_0 \Bigr)^\alpha.
\end{align*}
Continuity in $\alpha$ follows from the dominated convergence theorem since
\[
(f/\fbest)^\alpha\leq I_{\{ f \leq\fbest\}} + (f / \fbest) I_{\{f
> \fbest\}},
\]
where $I_{\{ \cdot\}}$ is the indicator function.
\end{proof}

The following theorem is an immediate consequence of the minimax
theorem and Lemma~\ref{lemma_h}.

\begin{proposition}
\label{minimax_h}
For any convex set $ A \subset\Fconv$,
\[
\inf_{0 \leq\alpha\leq1} \sup_{f \in A} \hstar_\alpha(\ftrue,
f) = \sup_{f \in A} \inf_{0 \leq\alpha\leq1} \hstar_\alpha
(\ftrue, f)
\]
\end{proposition}

Another useful application of the minimax theorem is the following.

\begin{proposition}
\label{minimax_g}
For any convex set $A \subset\FF$ and $f \in A$, define:
\[
g (\alpha,f) =
\begin{cases}
\Kstar(\ftrue, f) & \text{if $\alpha=0$}, \\
(1-\hstar_\alpha(\ftrue, f))/\alpha& \text{if $\alpha\in(0,1)$},
\\
1 - \int(f / \fbest) \ftrue\,d\mu& \text{if $\alpha=1$}.
\end{cases}
\]
Then
%\[ \inf_{0 \leq\alpha\leq1} \sup_{f \in A} g(\alpha,f) = \sup_{f
%\in A}\inf_{0 \leq\alpha\leq1} g(\alpha,f) \]
%
\[
\sup_{0 \leq\alpha\leq1} \inf_{f \in A} g(\alpha,f) = \inf_{f
\in A}\sup_{0 \leq\alpha\leq1} g(\alpha,f).
\]
\end{proposition}

\begin{proof}
On $A$, we give the $L_1(\mu_0)$-topology. From Lemma~\ref{lemma_h}
it follows that for each $\alpha\in(0,1)$, $ g(\alpha,f )$ is
continuous in $f$. Next, we argue that $ g(\alpha,f) $ is lsc in $f$
when $\alpha=0 $, i.e., we need to show that, if $\int|f_k - f| d\mu
_0 \to0$, for $f \in A$, then $\liminf\Kstar(\ftrue, f_k) \geq
\Kstar(\ftrue, f)$. Suppose, on the contrary, that $\liminf\Kstar
(\ftrue, f_k) = \delta< \Kstar(\ftrue, f)$. Then, there exists a
subsequence $\{f_{k^\prime}\}$ such that, $\Kstar(\ftrue,
f_{k^\prime})$ is increasing and $\lim\Kstar(\ftrue, f_{k^\prime})
= \delta$. Let $\alpha_n$ decrease to 0, and set $A_n = \{f \in A:
g(\alpha_n,f) \leq\delta\}$. By Lemma~\ref{KVlem6}, $g(\alpha_n,
f) \leq\Kstar(\ftrue, f)$ and hence for each $n$, $\{f_{k^\prime}
\} \subset A_n$. Further, the continuity of $g(\alpha_n,f)$ with
respect to $f$ implies that $f$ itself is in $A_n$. Thus $f \in\bigcap
_n A_n$, which implies $\Kstar(\ftrue,f) \leq\delta$, a
contradiction. Continuity at $\alpha=1$ is trivial. Similarly,
continuity of $g(\alpha, f)$ in $\alpha$ for $\alpha\in(0,1)$
follows from Lemma~\ref{lemma_h} and at $\alpha=0$ from Lemma~\ref
{KVlem6}. Finally, $g(\alpha,f)$ is convex in $f$ and by monotonicity
in $\alpha$, quasi-concave in $\alpha$. Applying the minimax theorem
gives the result.
\end{proof}

\begin{proof}[Proof of Proposition~\ref{KVKLrelation}]
(iii) $\implies$ (ii) is immediate. Now suppose (ii) holds, i.e.,
\[
\sup_{f \in A} \inf_{0\leq\alpha\leq1} \hstar_\alpha(\ftrue, f)
< e^{-\delta} \text{ for some }\delta>0.
\]
This means for a given, $f\in A$, $\exists\ 0<\alpha\leq1$ such that
$\hstar_\alpha(\ftrue, f) < e^{-\delta}$. Now, for any $\alpha
^\prime <\alpha$, setting $ p = \frac{\alpha}{\alpha^\prime}, q =
\frac{p}{p-1} $ and applying
Holder's inequality to $ (\frac{f}{\fbest})^{\alpha
^\prime}1 $ with respect to the measure $\ftrue d\mu$,
\[
\int\left(\frac{f}{\fbest}\right)^{\alpha^\prime}f_0 d\mu \leq
\left[ \int\left(\frac{f}{\fbest}\right)^{\alpha} f_0 d\mu
\right]^\frac{\alpha^\prime}{\alpha},
\]
or $ \hstar_{\alpha^\prime}(\ftrue, f) \leq [\hstar_{\alpha
}(\ftrue, f) ]^\frac{\alpha^\prime}{\alpha}$.

Consequently, for any $\alpha^\prime\leq\alpha, \hstar_{\alpha
^\prime}(\ftrue, f) \leq e^{-\alpha^\prime\delta}$. Equivalently,
for all $\alpha^\prime< \alpha$,
\[
\frac{1 - \hstar_{\alpha^\prime}(\ftrue,f)}{\alpha^\prime} \geq
\frac{1-e^{-\alpha^\prime\delta}}{\alpha^\prime}.
\]

As $ \alpha^\prime\downarrow0$, from Lemma~\ref{KVlem6} the
left-hand side of the last expression converges to $\Kstar(f_0,f)$ and
the right-hand side converges to $ -\frac{d}{d\alpha^\prime}
e^{-\alpha^\prime\delta}|_{\alpha^\prime=0} = \delta$. This holds
for each $ f \in A,$ and hence (i) holds.
This completes the proof of $(iii)\implies(ii) \implies(i)$.

We will now show that $(i)\implies(iii)$ when $A$ is a convex set,
thus concluding that in this case, the three conditions are equivalent.
Suppose (i) holds, then from Lemma~\ref{minimax_g},

\[
\sup_{0 \leq\alpha\leq1} \inf_{f \in A} g(\alpha,f) = \inf_{f
\in A}\sup_{0 \leq\alpha\leq1} g(\alpha,f) > \delta.
\]

Since $ \inf_{f \in A} g(\alpha,f)$ is increasing as $\alpha
\downarrow0$,
given any $ \delta^\prime< \delta$, there is a $\alpha_0 > 0$, such
that for all $f \in A$,

\[
g(\alpha_0,f ) = \frac{1 - \hstar_{\alpha_0}(f_0,f)}{\alpha_0} >
\delta^\prime.
\]

So that
$ \hstar_{\alpha_0}(f_0,f) < 1 - \alpha_0 \delta^\prime \leq
e^{-\eta} \text{ where } \eta= \alpha_0 \delta^\prime$.
Therefore, condition~(iii) holds.
\end{proof}

%s8 ###
\section{Illustrating need for assumptions on topology}
\label{illust.eg}

\begin{example}
This example shows that while $\fbest$ in the $L_1$ support of $\Pi$
is necessary for consistency, this may not follow from Assumption~\ref
{as:prior.concentration}.

Let $\ftrue$ be the $\unif(0,1)$ density and $\fbest$ be the $\unif
(0,2)$ density. Let $\FF_0=\{f_k: k \geq1\}$, where
\[
f_k(y) =
\begin{cases}
b_k & \text{if $y \in(0,1)$}, \\
2(1-b_k), & \text{if $y \in(1,3/2)$},
\end{cases}
\]
with $b_k \uparrow1/2$. Take any prior such that $\Pi(f_k)>0$ for all
$k$. Then Assumption~\ref{as:prior.concentration} is seen to be
satisfied because $\Kstar(\ftrue,f_k) \downarrow\Kstar(\ftrue
,\fbest)=\log2$. However, the $L_1$-distance does not vanish, i.e.,
$\|f_k-\fbest\|>1/4$ for all $k$.
\end{example}

\ifthenelse{1=1}{}{
\begin{example}
This example demonstrates that even if $\fbest$ is in the $L_1$
support of $\Pi$ and Assumption~\ref{as:prior.concentration} is
satisfied, consistency still may not hold.

For $b\in(0,\frac{1}{2})$ and $a\in(0,1)$, let $f_{b}$ and $g_{a}$
be densities with respect to the Lebesgue measure on [0,2], defined as follows:
\begin{align*}
f_b(y) & =
\begin{cases}
b & \text{if $y \in[0,1]$}, \\
1-b & \text{if $y \in[2,3]$},
\end{cases}
\\
g_a(y) & =
\begin{cases}
\frac{1}{2} & \text{if $y \in[0,a]$}, \\
1-\frac{a}{2} & \text{if $y \in[1,2]$}.
\end{cases}
\end{align*}
Let $\FF_0= \{g_a, f_b : a \in(0,1), \, b\in(0, \frac{1}{2}) \}$
endowed with the $L_1$ norm $\| \cdot\|$. Take the following priors on
$f_{b}$ and $g_{a}$:
\begin{align*}
\pi(f_b) & = \frac{1}{2\sqrt{2}}\Bigl(\frac{1}{2}-b\Bigr)^{-\frac
{1}{2}}, \quad b\in(0,1/2) \\
\pi(g_a) & = \frac{1}{4} a^{-\frac{1}{2}}, \quad a\in(0,1].
\end{align*}
Note that $\int_{0}^{\frac{1}{2}} \pi(f_b) db =\frac{1}{2}$ and
$\int_{0}^{1} \pi(g_a) da = \frac{1}{2}$.

Let $\ftrue$ be the density on (0,1) whose distribution function $F_0$ is
\[
F_{0}(y) =
\begin{cases}
2y \Bigl[ 1 - \Bigl(-\log(1 - 1/2^{1/2})\Bigr)^{-1/2} \Bigr] &
\text{if $y \in(0,1/2]$}, \\
1 - \bigl( -\log(1 - y^{1/2}) \bigr)^{-1/2} & \text{if $y \in(1/2, 1)$}.
\end{cases}
\]
We will see later that, for this $\ftrue$, we get $1-M_n^{1/2}<
e^{-n}$ $\prob_0$-almost surely for all large $n$, where $M_n= \max\{
Y_1,\ldots,Y_{n}\}$. Let $\fbest$ be the $\unif(0,2)$, i.e.,
$g_a(\cdot)$ for $a=1$. Then it is easy to see the following:
\begin{align*}
\int\log\Bigl( \frac{\fbest}{f_b} \Bigr) f_0 \,d\mu& = -\log
(2b) \\
\int\log\Bigl( \frac{\fbest}{g_a} \Bigr) f_0 \,d\mu& =
\begin{cases} \infty& \text{if $a \in(0,1)$}, \\ 0 & \text{if
$a=1$}.
\end{cases}
\end{align*}
These show that, indeed, $\fbest$ minimizes the Kullback--Leibler
divergence from $\ftrue$. Also, $\{f: \Kstar(\ftrue, f) < \eps\}$
contains $\{f_b: -\log(2b) < \eps\}$. Clearly, the assumed prior puts
positive mass on the latter set, so Assumption~\ref
{as:prior.concentration} is satisfied. Further, note that, for $b \in
(0,1/2)$ and $a \in(0,1)$,
\[
\|f_b - \fbest\| > 1/2 \quad\text{and} \quad\|g_a - \fbest\| =
1-a.
\]
Therefore, for any $\eps\in(0,1/2)$, we have:
\begin{align*}
\Pi_n(\|f-\fbest\| > \eps) & = \frac{\int_{M_n}^{1-\eps} 2^{-n}
\, \frac14 a^{-1/2} \,da + \int_0^{1/2} b^n \, \frac{1}{2\sqrt{2}}
(\frac12 - b)^{-1/2} \,db}{\int_{M_n}^1 2^{-n} \, \frac14 a^{-1/2} \,
da + \int_0^{1/2} b^n \, \frac{1}{2\sqrt{2}} (\frac12 - b)^{-1/2} \,
db} \\
& \geq\frac{0 + A_n(M_n)}{1 + A_n(M_n)},
\end{align*}
where
\[
A_n(m) = \frac{\int_0^{1/2} b^n \, \frac{1}{2\sqrt{2}}(\frac12 -
b)^{-1/2} \,db}{\int_m^1 2^{-n} \, \frac14 a^{-1/2} \,da},
\]
and $M_n=\max\{Y_1,\ldots,Y_n\}$ as before. We can simplify $A_n$ as follows:
\[
A_n(m) = \frac12 \frac{B(n+1, 1/2)}{1-m^{1/2}},
\]
where $B(\cdot,\cdot)$ is the usual beta function. We claim that
$1-M_n^{1/2} < e^{-n}$ for all large $n$ with $\prob_0$-probability~1.
To see this, write:
\[
\prob_0(1-M_n^{1/2} > e^{-n}) = \prob_0\{M_n < (1-e^{-n})^2\} = \Bigl
(1-\frac{1}{\sqrt{n}}\Bigr)^n \leq e^{-\sqrt{n}}.
\]
Since this upper bound is summable over $n \geq1$, the claim follows
from the Borel--Cantelli lemma. Therefore, when $n$ is large,
\begin{align*}
A_n(M_n) & \geq(1/2) e^n B(n+1, 1/2) \\
& \geq(1/2) e^n \int_0^1 x^n (1-x)^{-1/2} \,dx \\
& > \int_{2/e}^1 (ex)^n (1-x)^{-1/2} \,dx \\
& > 2^n \int_{2/e}^1 (1-x)^{-1/2} \,dx,
\end{align*}
and this implies that $\liminf A_n(M_n) = \infty$ $\prob_0$-almost
surely. Consequently, $\Pi_n(\|f-\fbest\| > \eps) \not\to0$, i.e.,
Assumption~\ref{as:prior.concentration} together with $\fbest$ in the
$L_1$ support of $\Pi$ is not enough to guarantee $L_1$ consistency.
\end{example}
}

\begin{example}
This example demonstrates that even if $\fbest$ is in the $L_1$
support of $\Pi$ and Assumption~\ref{as:prior.concentration} is
satisfied, consistency still may not hold.

Let $\mu$ be the measure obtained as a sum of the Lebesgue measure on
$[0,2]$ and point masses on integers $k \geq3$. For $k \geq3$ and
$a\in(0,1]$, let $f_{k}$ and $g_{a}$ be densities with respect to the
measure $\mu$, defined as follows:
\begin{align*}
f_k(y) & =
\begin{cases}
\frac{1}{2}-\frac{1}{k} & \text{if $y \in[0,1]$}, \\
\frac{1}{2}+\frac{1}{k} & \text{if $y=k$},
\end{cases}
\\
g_a(y) & =
\begin{cases}
\frac{1}{2} & \text{if $y \in[0,a]$}, \\
1-\frac{a}{2} & \text{if $y \in[1,2]$}.
\end{cases}
\end{align*}
Let $\FF= \{g_a, f_k : a \in(0,1), \, k\geq3 \}$ endowed with the
$L_1(\mu)$ norm $\| \cdot\|$. Take the following priors on $f_{k}$
and $g_{a}$:
\begin{align*}
\pi(f_k) & = \frac{1}{2^{k-1}}, \quad k\geq3, \\
\pi(g_a) & = \frac{1}{4} a^{-\frac{1}{2}}, \quad a\in(0,1].
\end{align*}
Note that $\sum_{k\geq3} \pi(f_k) =\frac{1}{2}$ and $\int_{0}^{1}
\pi(g_a) da = \frac{1}{2}$.

Let $\ftrue$ be the density on $(0,1)$ whose distribution function $F_0$ is
\[
F_{0}(y) =
\begin{cases}
2y \Bigl[ 1 - \Bigl(-\log(1 - 1/2^{1/2})\Bigr)^{-1/2} \Bigr] &
\text{if $y \in(0,1/2]$}, \\
1 - \bigl( -\log(1 - y^{1/2}) \bigr)^{-1/2} & \text{if $y \in(1/2, 1)$}.
\end{cases}
\]
We will see later that, for this $\ftrue$, we get $1-M_n^{1/2}<
e^{-n}$ $\prob_0$-almost surely for all large $n$, where $M_n= \max\{
Y_1,\ldots,Y_{n}\}$. Let $\fbest$ be the $\unif(0,2)$, density,
i.e., $g_a(\cdot)$ for $a=1$. Then it is easy to see the following:
\begin{align*}
\int\log\Bigl( \frac{\fbest}{f_k} \Bigr) f_0 \,d\mu& = \log
\left(\frac{1/2}{1/2-1/k}\right), \\
\int\log\Bigl( \frac{\fbest}{g_a} \Bigr) f_0 \,d\mu& =
\begin{cases} \infty& \text{if $a \in(0,1)$}, \\ 0 & \text{if
$a=1$}.
\end{cases}
\end{align*}
These show that, indeed, $\fbest$ minimizes the Kullback--Leibler
divergence from $\ftrue$. Also, $\{f: \Kstar(\ftrue, f) < \eps\}$
contains $\{f_k: \log(\frac{1/2}{1/2-1/k}) < \eps\}$
$=\{f_k: k \geq2(1-e^{- \eps})^{-1}\}$. Clearly, the assumed prior
puts positive mass on the latter set, so Assumption~\ref
{as:prior.concentration} is satisfied. Further, note that, for $k \geq
3$ and $a \in(0,1)$,
\[
\|f_k - \fbest\| > 1/2 \quad\text{and} \quad\|g_a - \fbest\| =
1-a.
\]
Therefore, for any $\eps\in(0,1/2)$, we have
\begin{align*}
\Pi_n(\|f-\fbest\| > \eps) & = \frac{\int_{M_n}^{1-\eps} 2^{-n}
\, \frac14 a^{-1/2} \,da + \sum_{k\geq3} (\frac{1}{2}-\frac
{1}{k})^n \, \frac{1}{2^{k-1}}}{\int_{M_n}^{1} 2^{-n} \, \frac14
a^{-1/2} \,da + \sum_{k\geq3} (\frac{1}{2}-\frac{1}{k})^n \, \frac
{1}{2^{k-1}}} \\
& \geq\frac{0 + A_n(M_n)}{1 + A_n(M_n)},
\end{align*}
where
\[
A_n(m) = \frac{\sum_{k\geq3} (\frac{1}{2}-\frac{1}{k})^n \, \frac
{1}{2^{k-1}}}{\int_m^1 2^{-n} \, \frac14 a^{-1/2} \,da},
\]
and $M_n=\max\{Y_1,\ldots,Y_n\}$ as before. We claim that
$1-M_n^{1/2} < e^{-n}$ for all large $n$ with $\prob_0$-probability~1.
To see this, write
\[
\prob_0(1-M_n^{1/2} > e^{-n}) = \prob_0\{M_n < (1-e^{-n})^2\} = \Bigl
(1-\frac{1}{\sqrt{n}}\Bigr)^n \leq e^{-\sqrt{n}}.
\]
Since this upper bound is summable over $n \geq1$, the claim follows
from the Borel--Cantelli lemma. Therefore, when $n$ is large,
\begin{align*}
A_n(M_n) & \geq\sum_{k\geq3} e^{n}\,\left(\frac{1}{2}-\frac
{1}{k}\right)^n \, \frac{1}{2^{k-1}}.
\end{align*}
Since for large enough $k$, $e \, (\frac{1}{2}-\frac
{1}{k}) > \frac{e}{2.5} $, this implies that $\liminf A_n(M_n)
= \infty$ $\prob_0$-almost surely. Consequently, $\Pi_n(\|f-\fbest
\| > \eps) \not\to0$, i.e., Assumption~\ref
{as:prior.concentration}, together with $\fbest$ in the $L_1$ support
of $\Pi,$ is not enough to guarantee $L_1$-consistency.
\end{example}
%
%s9 ###
\section{Proofs of results for the $i.n.i.d.$ case}
\label{inidproofs}
Here, we provide details of proofs of results for the $i.n.i.d.$ case
discussed in Section~\ref{S:inid}.
\begin{proof}[Proof of Lemma~\ref{lemnhbd}]
Since Assumption A holds, by Arzel\`a--Ascoli theorem, we have the following:
\begin{itemize}
\item[{(i)}] $\Theta$ is uniformly bounded, i.e., $\exists\ M$ such
that $|\theta({\bf x})|\leq M$ $\ \forall\ \theta\in\Theta$ and
${\bf x}\in\mathcal{X}$.
\item[{(ii)}]$\Theta$ is equi-uniformly-continuous, i.e., for ${\bf
x}_0\in\mathcal{X}$, given $\epsilon>0$, $\exists\ \delta>0$ such
that $\forall\ {\bf x} : \ \|{\bf x}-{\bf x}_0\|<\delta$, $|\theta
_{\bf x}-\theta_{{\bf x}_0}|$ $<\epsilon$, $\ \forall\ \theta\in
\Theta$.
\end{itemize}

Without loss of generality, for $\theta'\in U^c$, we have $\theta
'_{{\bf x}_0}-\theta^*_{{\bf x}_0}>\epsilon$. By (ii) above, i.e.,
equicontinuity, $\exists\ \delta'$ such that $\forall\ \|{\bf
x}-{\bf x}_0\|<\delta'$, we have $|\theta_{\bf x}-\theta_{{\bf
x}_0}|<\frac{\epsilon}{4}$, $\forall\ \theta\in\Theta$\xch{.}{ .} In
particular, for such ${\bf x}$, $|\theta^*_{{\bf x}}-\theta^*_{{\bf
x}_0}|<\frac{\epsilon}{4}$. Therefore,
\begin{eqnarray*}
\hspace*{50pt}\theta'_{\bf x}-\theta^*_{\bf x} = \theta'_{\bf x}-\theta'_{{\bf
x}_0} + \theta'_{{\bf x}_0}-\theta^*_{{\bf x}_0} +\theta^*_{{\bf
x}_0}-\theta^*_{{\bf x}}\geq -\frac{\epsilon}{4} + \epsilon- \frac
{\epsilon}{4} = \frac{\epsilon}{2}.\hspace*{42pt}\qedhere
\end{eqnarray*}
\end{proof}

\begin{proof}[Proof of Proposition~\ref{denominid}]
Due to the compactness of $\Theta\times[-M, M]^2$ (Assumption~A), and
continuity (Assumption D), it follows that $E_{\bf x}\log\frac
{f_{t}}{f_{t^\prime}}$ is uniformly continuous in $(x, t, t^\prime)$. Hence,
the collection $\{ \E_{\bf x}\log\frac{f_{\theta^{*}({\bf
x}_i)}}{f_{\theta({\bf x}_i)}}, \ i\geq1\} $ is equicontinuous
w.r.t. $\theta\in\Theta$. Further, Assumption D implies that $
\{ \E_{\bf x}\log^2\frac{f_{\theta^{*}({\bf x}_i)}}{f_{\theta({\bf
x}_i)}}, \ i\geq1\}$ is uniformly bounded. Hence, $\exists\
\delta\in(0,1)$ such that
\begin{eqnarray*}
&&\left\{\sup_{{\bf x}\in\mathcal{X}}|\theta({\bf x})-\theta
_1({\bf x})|<\delta\right\}\\
&& \subseteq V_{\epsilon}= \left\{ \theta: \sup_{i\geq1} \E_{{\bf
x}_i}\log\frac{f_{\theta^*_{{\bf x}_i}}(Y_i)}{f_{\theta_{{\bf
x}_i}}(Y_i)}<\epsilon, \ \sum_{i=1}^{\infty} \frac{1}{i^2}\E_{{\bf
x}_i}\left(\log\frac{f_{\theta^*_{{\bf x}_i}}(Y_i)}{f_{\theta
_{{\bf x}_i}}(Y_i)}\right)^2<\infty\right\}.
\end{eqnarray*}
Assumption C will therefore ensure that the prior gives positive mass
for the set $V_{\epsilon}$. Now, observing that $R'_{2n}\geq\int
_{V_{\epsilon}}e^{\sum_{i=1}^{n}\log(\frac{f_{\theta_{{\bf
x}_i}}(Y_i)}{f_{\theta^*_{{\bf x}_i}}(Y_i)} )}d\Pi(\theta) $
and an application of the strong law of large numbers for independent
random variables leads to
\[
\sum_{i=1}^{n}\log\left(\frac{f_{\theta_{{\bf
x}_i}}(Y_i)}{f_{\theta^*_{{\bf x}_i}}(Y_i)} \right)>-2n\epsilon\ \
a.s\xch{.}{..}
\]
Rest\, of\, the\, proof\, is\, along\, the\, lines\, of\, Lemma 4.4.1 of \cite{Ghosh_RVR2003}.
\end{proof}
%
%~\\
%
\begin{proof}[Proof of Lemma~\ref{sepcondinid}]
For $\theta'\in U^{c}$, let ${\bf x}_0$ be such that $|\theta'({\bf
x}_0)-\theta^*({\bf x}_0)|>\epsilon$. Then by Lemma~\ref{lemnhbd},
$\exists\ \delta'$ such that $\forall\ {\bf x} \in A_{{\bf
x}_0,\delta'}:= \{ {\bf x} \ : \ \|{\bf x}-{\bf x}_0\|<\delta'\}$, we
have $|\theta'_{\bf x}-\theta^*_{\bf x}| \geq\frac{\epsilon}{2}$.
Therefore, by Assumption E, $\exists\ \delta\in(0,1)$ such that $\E
_{{\bf x}}\log\frac{f_{\theta^*_{\bf x}}}{f_{\theta'_{{\bf x}}}}
\geq\delta$ for all ${\bf x} \in A_{{\bf x}_0, \delta'}$.

For $({\bf x},t,t') \in\mathcal{X}\times[-M,M]^2$, let $ g_{\alpha
}({\bf x},t,t'):= \frac{1-\E_{\bf x}( \frac
{f_{t}}{f_{t'}})^{\alpha}}{\alpha}$. By Lemma~\ref{KVlem6},
we have that, $g_{\alpha}({\bf x},t,t') $ increases to $\E_{\bf
x}\log\frac{f_{t'}}{f_{t}} $ as $\alpha\downarrow0$. By Assumption
D, both $g_\alpha(\cdot,\cdot,\cdot)$ and the limiting function are
continuous in $({\bf x},t,t')$, which is in the compact set $\mathcal
{X}\times[-M,M]^2$. Hence, it follows by Dini's theorem that this
convergence is uniform, i.e.,
\[
\lim_{\alpha\downarrow0} \frac{1-\E_{\bf x}\left( \frac
{f_{t}}{f_{t'}}\right)^{\alpha}}{\alpha}\uparrow\E_{\bf x}\log
\frac{f_{t'}}{f_{t}} \ \mbox{uniformly on } \mathcal{X}\times
[-M,M]^2.
\]
Let $\kappa:=\kappa({\bf x}_0,\delta')$ as in Assumption B. Then,
$\exists\ 0<\alpha'<1$ such that $g_{\alpha'}({\bf x},t,t')> \E
_{\bf x}\log\frac{f_{t'}}{f_{t}}- \kappa\frac{\delta}{2}, \forall
\ ({\bf x},t,t')\in\  \mathcal{X}\times[-M,M]^2$. In particular,
$g_{\alpha'}({\bf x}_i,\theta_{{\bf x}_i},\theta^*_{{\bf x}_i})\geq
\E_{{\bf x}_i}\log\frac{f_{\theta^*_{{\bf x}_i}}}{f_{\theta_{{\bf
x}_i}}}-\kappa\frac{\delta}{2}$ $\forall\ i\geq 1, \theta\in
\Theta$. Also, in general $\E_{{\bf x}_i}\log\frac{f_{\theta
^*_{{\bf x}_i}}}{f_{\theta_{{\bf x}_i}}}\geq0$. Combining this with
the observation we made at the beginning of the proof that $\E_{{\bf
x}}\log\frac{f_{\theta^*_{\bf x}}}{f_{\theta'_{{\bf x}}}} \geq
\delta$ $ \forall\ x\in A_{x_0, \delta'}$, we get
%
%e10 ###
\begin{equation}
\label{stepp}
g_{\alpha'}({\bf x}_i,\theta'_{{\bf x}_i},\theta^*_{{\bf x}_i})\geq
\delta\cdot I_{A_{{\bf x}_0,\delta'}}({\bf x}_i)-\kappa\frac{\delta}{2},
\end{equation}
where $I_{A_{{\bf x}_0,\delta'}}({\bf x})$ is the indicator function
which is $1$ when ${\bf x}\in A_{{\bf x}_0, \delta'}$ and $0$
otherwise. Note that, by Assumption B, for sufficiently large n, $\frac
{1}{n}\sum_{i=1}^{n}I_{A_{{\bf x}_0, \delta'}}({\bf x}_i) > \frac
{3\kappa}{4}$. Using this, along with a bit of algebra on (\ref
{stepp}), we can conclude that the following inequality holds for
sufficiently large $n$:
\begin{equation*}
\E_0\left(\prod_{i=1}^{n} \frac{f_{\theta'_{{\bf
x}_i}}(Y_i)}{f_{\theta^*_{{\bf x}_i}}(Y_i)}\right)^{\alpha'} \leq
e^{-\delta\sum_{i=1}^{n}\cdot I_{A_{{\bf x}_0,\delta'}}({\bf
x}_i)+n\kappa\frac{\delta}{2}. } \leq e^{-n\kappa\frac{\delta}{4}}.
\end{equation*}
The result follows by assigning $\delta_1:= \kappa\frac{\delta}{4}$.
\end{proof}

\begin{proof}[Proof of Proposition~\ref{numerinid}]
First, we claim by Assumption D that the collection of functions $
\{ \E_{\bf x}\frac{f_{\theta({\bf x}_i)}}{f_{\theta'({\bf x}_i)}},
\ i\geq1\} $ is equicontinuous w.r.t. the sup-norm metric on
$\Theta$.
Note that $\E_{\bf x}\frac{f_t}{f_{t'}}$ is a continuous function on
a compact set $\mathcal{X}\times[-M,M]^2$. Hence, it is uniformly
continuous. So, given $\epsilon>0$, $\exists\  \delta$ such that if
$\|{\bf x}-{\bf x}_1\|<\delta$, $|t-t_1|<\delta$ and
$|t'-t'_1|<\delta$ then $\vert\E_{{\bf x}_1}\frac
{f_{t_1}}{f_{t'_1}}-\E_{\bf x}\frac{f_t}{f_{t'}} \vert
<\epsilon$. In particular, let $\theta, \theta_1 \in\Theta$ be
such that $\sup_{{\bf x}\in\mathcal{X}}|\theta({\bf x})-\theta
_1({\bf x})|<\delta$. Then for any ${\bf x} \in\mathcal{X}$, taking
${\bf x}_1={\bf x}$, $t'=t'_1=\theta({\bf x})$ and $t=\theta({\bf
x})$, $t_1=\theta_1({\bf x})$, we get $\vert\E_{\bf x}\frac
{f_{\theta({\bf x})}}{f_{\theta'({\bf x})}}- \E_{\bf x}\frac
{f_{\theta_1({\bf x})}}{f_{\theta'({\bf x})}}\vert<\epsilon
$. Hence the collection of functions $\{ \E_{\bf x}\frac
{f_{\theta({\bf x}_i)}}{f_{\theta'({\bf x}_i)}}, \ i\geq1\} $
is equicontinuous in $\theta$ w.r.t. sup-norm metric.

Define $A_{\theta'}:= \{\theta\in\Theta: \ \E_{{\bf
x}_i}[\frac{f_{\theta_{{\bf x}_i}}}{f_{\theta'_{{\bf x}_i}}}
]< e^{\frac{\delta}{2}} ,\forall i\geq1\}$. This set
clearly contains $\theta'$ and it is an open set due to
equicontinuity. By Lemma~\ref{sepcondinid}, $\exists\ \alpha'\in
(0,1)$ such that
\[
\E_0 \left(\prod_{i=1}^{n}\frac{f_{\theta'_{{\bf
x}_i}}(Y_i)}{f_{\theta^*_{{\bf x}_i}}(Y_i)} \right)^{\alpha'} <
e^{-n\alpha' \delta}\ \ \mbox{ for all sufficiently large } n.
\]
Let $\alpha=\alpha'/2$. Then, for sufficiently large $n$,
\begin{eqnarray*}
\hspace*{27pt}&&\E_0\left[ \left(\int_{A_{\theta'}}\prod_{i=1}^{n}\frac
{f_{\theta_{{\bf x}_i}}(Y_i)}{f_{\theta^*_{{\bf x}_i}}(Y_i)} d\nu
(\theta) \right)^{\alpha}\right]\\
&&= \E_0\left[\left( \frac{f_{\theta'_{{\bf x}_i}}(Y_i)}{f_{\theta
^*_{{\bf x}_i}}(Y_i)} \right)^{\alpha} \left(\int_{A_{\theta
'}}\prod_{i=1}^{n}\frac{f_{\theta_{{\bf x}_i}}(Y_i)}{f_{\theta
'_{{\bf x}_i}}(Y_i)} d\nu(\theta) \right)^{\alpha}\right] \\
&& \mbox{(By Cauchy--Schwartz inequality)}\\
&&\leq\left(\E_0\left[\left( \frac{f_{\theta'_{{\bf
x}_i}}(Y_i)}{f_{\theta^*_{{\bf x}_i}}(Y_i)} \right)^{2\alpha} \right
]\right)^{\frac{1}{2}} \cdot\left(\E_0\left[\left(\int
_{A_{\theta'}}\prod_{i=1}^{n}\frac{f_{\theta_{{\bf
x}_i}}(Y_i)}{f_{\theta'_{{\bf x}_i}}(Y_i)} d\nu(\theta) \right
)^{2\alpha}\right]\right)^{\frac{1}{2}} \\
&&\mbox{(By Jensen's inequality)}\\
&&\leq\left(\E_0\left[\left( \frac{f_{\theta'_{{\bf
x}_i}}(Y_i)}{f_{\theta^*_{{\bf x}_i}}(Y_i)} \right)^{\alpha'} \right
]\right)^{\frac{1}{2}} \cdot\left(\int_{A_{\theta'}}\E_0\left
[\prod_{i=1}^{n}\frac{f_{\theta_{{\bf x}_i}}(Y_i)}{f_{\theta'_{{\bf
x}_i}}(Y_i)}\right] d\nu(\theta) \right)^{\frac{\alpha'}{2}} \\
&& < e^{-n\alpha'\frac{\delta}{2}} \cdot e^{n\alpha'\frac{\delta
}{4}} = e^{-n\alpha\frac{\delta}{2}}.\hspace*{200pt}\qedhere
\end{eqnarray*}
\end{proof}

%\bibliographystyle{ba}
%\bibliography{misspecbib}

\begin{acknowledgement}
We thank the Editor, Associate Editor and the Referees for their
insightful comments and helping us significantly improve the paper. The
work of Karthik Sriram was supported by a research grant provided by
Indian Institute of Management Ahmedabad, India.
\end{acknowledgement}

\end{document}